\documentclass[10pt,oneside,a4paper]{amsart}
\usepackage{amsmath,amsfonts,amsthm, amssymb}
\usepackage{array}
\usepackage[all,textures]{xy}
\usepackage{portland}
\usepackage{color}
\usepackage[orange]{xcolor}

\theoremstyle{plain}
\newtheorem{theorem}{Theorem}[section]
\newtheorem{proposition}[theorem]{Proposition}
\newtheorem{lemma}[theorem]{Lemma}
\newtheorem{corollary}[theorem]{Corollary}

\theoremstyle{definition}
\newtheorem{definition}[theorem]{Definition}

\theoremstyle{remark}

\newtheorem{remark}[theorem]{Remark}
\newtheorem{example}[theorem]{Example}

\newcommand{\ovl}{\overline}
\newcommand{\Kern}{\mathrm{Ker}}

\newcommand{\cone}{\mathrm{cone}}

\renewcommand{\lim}{\mathrm{lim}}
\newcommand{\colim}{\mathrm{colim}}

\newcommand{\per}{\mathsf{per}}

\newcommand{\tw}{\mathsf{tw}}

\newcommand{\Tw}{\mathsf{Tw}}

\newcommand{\Free}{\mathsf{Free}}

\newcommand{\Hom}{\mathrm{Hom}}

\newcommand{\RHom}{\mathrm{RHom}}

\newcommand{\Def}{\mathrm{Def}}

\newcommand{\embr}{\mathrm{embr}}

\newcommand{\Ob}{\mathrm{Ob}}
\newcommand{\Z}{\mathbb{Z}}

\newcommand{\N}{\mathbb{N}}

\newcommand{\AAA}{\mathfrak{a}}
\newcommand{\BBB}{\mathfrak{b}}

\newcommand{\CCC}{\mathfrak{c}}

\newcommand{\CC}{\mathbf{C}}

\newcommand{\Mod}{\ensuremath{\mathsf{Mod}} }

\newcommand{\lra}{\longrightarrow}

\CompileMatrices
\SelectTips{cm}{11}

\title{The curvature problem for formal and infinitesimal deformations}
\author{Wendy Lowen} 
\address[Wendy Lowen]{Departement Wiskunde-Informatica, Universiteit Antwerpen, Middelheimcampus,
Middelheimlaan 1,
2020 Antwerp, Belgium}
\email{wendy.lowen@uantwerpen.be}
\author{Michel Van den Bergh} 
\address[Michel Van den Bergh]{Departement WNI, Universiteit Hasselt,
3590 Diepenbeek, Belgium}
\email{michel.vandenbergh@uhasselt.be}
\thanks{The author acknowledges the support of the European Union for the ERC grant No 257004-HHNcdMir and the support of the Research Foundation Flanders (FWO) under Grant No. G.0112.13N.}
\thanks{The second author is a director of research at the Fund for Scientific Research, Flanders}

\begin{document}
\maketitle

\begin{abstract}
We interpret all Maurer-Cartan elements in the formal Hochschild complex of a small dg category which is cohomologically bounded above in terms of torsion Morita deformations. This solves the \emph{curvature problem}, i.e. the phenomenon that such Maurer-Cartan elements naturally parameterize \emph{curved} $A_{\infty}$-deformations. In the infinitesimal setup, we show how $n+1$-th order curved deformations give rise to $n$-th order uncurved Morita deformations.
\end{abstract}

\section{Introduction}

Let $k$ be a field. For flexibility with the Maurer Cartan formalism, we assume that $k$ is of characteristic zero.
Let $A_0 = (A_0, \mu^{(0)})$ be a dg $k$-algebra with Hochschild complex $\CC(A_0)$. It is well known that the shifted Hochschild complex $\Sigma \CC(A_0)$ has the structure of a $B_{\infty}$-algebra, in particular it is a dg Lie algebra \cite{getzlerjones}.

Consider
$$\mathrm{MC}(A_0; k[[t]]) = \{ \phi \in \Sigma\CC(A_0)[[t]]^1 \,\,| \,\, [\mu^{(0)}, \phi t] + \frac{1}{2}[\phi t,\phi t] = 0 \}/ \sim,$$
the set of elements $\phi$ of shifted degree $1$ for which $\phi t \in \Sigma\CC(A_0)[[t]]$ is a solution of the Maurer Cartan equation up to gauge equivalence of the elements $\phi t$. 
If $A_0$ is an ordinary $k$-algebra, there is a bijection
\begin{equation}\label{eqalg}
\Def_{\mathrm{alg}}(A_0; k[[t]]) \cong \mathrm{MC}(A_0; k[[t]])
\end{equation}
for the set $\Def_{\mathrm{alg}}(A_0; k[[t]])$ of formal algebra deformations of $A_0$ up to equivalence of deformations. Under this bijection, a representative $\phi = \sum_{k = 0}^{\infty} \phi^{(k)}t^k$ on the right hand side corresponds to the formal deformation $A = A_0[[t]]$ endowed with the multiplication $\mu = \mu^{(0)} + \phi t$ on the left hand side. 
Let us now look at a dg $k$-algebra $A_0$. In this case, inspection of the Hochschild complex shows that $\phi \in \CC^2(A_0)[[t]]$ corresponds to
$$(\phi_n)_{n \geq 0} \in \prod_{n \geq 0} \Hom^{2 - n}_k(A^{\otimes n}, A)[[t]]$$
wih $\phi_n = \sum_{k = 0}^{\infty} \phi_n^{(k)}t^k$. The corresponding deformed structure $\mu = \mu^{(0)} + \phi t$ on $A = A_0[[t]]$ is naturally a \emph{curved} $A_{\infty}$-algebra ($cA_{\infty}$-algebra) structure, with the curvature given by $$\mu_0 = \phi_0 t \in (A_0[[t]])^2.$$
In particular, the differential $\mu_1 = \mu^{(0)}_1 + \phi_1t$ does not square to zero, but instead satisfies the identity
$$\mu_1^2 + \mu_2(\mu_0, -) + \mu_2(-, \mu_0) = 0.$$
This places the object $(A, \mu)$ somewhat outside the classical setup of homological algebra, as an object with no cohomology groups, and no classical derived category. We refer to the fact that $\mathrm{MC}(A_0; k[[t]])$ parameterizes $cA_{\infty}$-deformations rather than dg or $A_{\infty}$-deformations as the \emph{curvature problem}. To solve the curvature problem, we try to find an alternative type of deformations, which can be understood within the realm of classical homological algebra and for which we obtain a bijection replacing \eqref{eqalg}.

To find such a type of deformations, we look at the natural invariance properties of the Hochschild complex. Let $A_0$ and $B_0$ be dg $k$-algebras and let $M_0$ be a Morita $A_0$-$B_0$-bimodule, that is, a bimodule inducing an equivalence of categories $- \otimes^L M_0: D(A_0) \lra D(B_0)$. In \cite{keller}, from the bimodule $M_0$, Keller constructs an isomorphism 
\begin{equation}\label{keller}
\Gamma_{M_0}: \Sigma\CC(A_0) \lra \Sigma \CC(B_0)
\end{equation}
in the homotopy category of $B_{\infty}$-algebras. Interestingly, this isomorphism does not restrict to the kernels of the respective characteristic curvature projections 
$$c_{A_0}: \Sigma \CC(A_0) \lra \Sigma A_0: \phi = (\phi_n)_{n \geq 0} \longmapsto c_{A_0}(\phi) = \phi_0$$
and $c_{B_0}$. In other words, it is possible that $\phi \in \Sigma\CC(A_0)$ has $c_{A_0}(\phi) \neq 0$ but $\Gamma_{M_0}(\phi)$ has $c_{B_0}(\Gamma_{M_0}(\phi)) = 0$. This creates the possibility of getting rid of curvature, up to changing to a Morita equivalent dg algebra.

By definition, a formal Morita deformation of $A_0$ consists of dg algebra $B_0$, an $A_0$-$B_0$-Morita bimodule $M_0$, and a formal dg deformation $(B = B_0[[t]], \mu)$ of $\mu^{(0)}$. Using \eqref{keller}, we obtain a natural map from formal Morita deformations of $A_0$ to $\mathrm{MC}(A_0; k[[t]])$. To make this map injective, for technical reasons we have to consider Morita deformations up to \emph{torsion} Morita equivalence. The relevant notion of torsion derived category, which is used in order to ensure the validity of a derived version of Nakayama's Lemma in the formal deformation framework, is developed in \S \ref{parpartor}. We thus arrive at an injective map
\begin{equation}\label{thetaintro}
\Theta: \Def_{\mathrm{tMor}}(A_0; k[[t]]) \lra \mathrm{MC}(A_0; k[[t]])
\end{equation}
replacing \eqref{eqalg}. Our main result, stated more generally in Theorem \ref{thmmain} for small dg categories, is the following:

\begin{theorem} \label{thmintro}
Let $A_0$ be a dg $k$-algebra.
Consider $\phi = (\phi_n)_{n \geq 0} \in \mathrm{MC}(A_0; k[[t]])$.  If $\phi^{(0)}_0$ is nilpotent in $H^{\ast}A_0$, then $\phi$ is in the image of $\Theta$.
\end{theorem}

\begin{corollary}
Let $A_0$ be a dg $k$-algebra which is cohomologically bounded above. The map $\Theta$ is bijective.
\end{corollary}

The necessity of the condition in the theorem is shown by the example of the graded field due to Keller, which was previously considered in \cite{kellerlowen}. For $A_0 = k[x, x^{-1}]$ with $\deg(x) = 2$, the element $\phi = x \in \mathrm{MC}(A_0; k[[t]])$ is shown in \S \ref{pargraded} not to be in the image of $\Theta$. On the other hand, part (2) of the theorem covers the geometric case. More precisely, if $X$ is a quasi-compact quasi-separated scheme, it was shown in \cite{bondalvandenbergh} that the derived category $D_{\mathrm{Qch}}(X)$ is compactly generated by a single perfect complex $M$, and thus equivalent to the derived category $D(A_0)$ of $A_0 = \RHom(M,M)$. The dg algebra $A_0$ is cohomologically bounded above.

Our approach to Theorem \ref{thmintro} is based upon the construction of uncurved twisted objects over a $k[[t]]$-linear $cA_{\infty}$-algebra $A$.  In \S \ref{parpartwist}, we introduce the $cA_{\infty}$-category $\Tw(A)$ of twisted objects over $A$. An object of $\Tw(A)$ is a formal expression $M = \oplus_{i \in I} \Sigma^{n_i} A$ with $n_i \in \Z$, endowed with a connection $\delta_M \in \Hom^1(M,M)$. Here, $\Hom$-modules are defined in the obvious way as modules of (limits of) infinite column finite matrices with entries in $A$. The $cA_{\infty}$-structure on $\Tw(A)$ is given by
\begin{equation}\label{eqembrin}
\embr_{\delta}(\mu) = \sum_{n = 0}^{\infty} \mu\{ \delta^{\otimes n}\} = \mu + \mu\{\delta\} + \mu\{\delta, \delta\} + \dots
\end{equation}
for the brace algebra structure on the shifted Hochschild object $\Sigma \CC(\Tw(A))$. Here, the notation $\embr_{\delta}(\mu)$ stands for ``$\mu$ embracing $\delta$''.
Convergence of \eqref{eqembrin} is understood in a ``pointwise $t$-adic'' way, and in order to ensure convergence, we require connections to be \emph{locally $\mu$-allowable} (see Definition \ref{defnilp}). This condition is weaker than the usual upper triangular condition used in the definition of classical (uncurved) twisted objects over a dg or $A_{\infty}$-algebra.

Let $(A_0, \mu^{(0)})$ be a strictly unital $k$-linear $A_{\infty}$-algebra, and consider, for $\phi \in \mathrm{MC}(A_0; k[[t]])$, the $cA_{\infty}$-structures $\mu = \mu^{(0)} + \phi t$ on $A = A_0[[t]]$ and $\eta = \embr_{\delta}(\mu)$ on $\Tw(A)$ as in \eqref{eqembrin}. 
We may suppose that $\mu$, and hence also $\eta$, is strictly unital. The curvature of $(A, \mu)$ is given by
\begin{equation}\label{eqfact}
\mu_0 = \mu^{(0)}_0 + \phi_0 t = \phi_0 t.
\end{equation}
Our main technical tool is the construction of the following ``two-sided cone'' in $\Tw(A)$, associated to the factorization \eqref{eqfact} of $\mu_0$. We define $A_t = A \oplus \Sigma^{-1}A \in \Tw(A)$ to be endowed with the connection
$$\delta_{A_t} = \begin{pmatrix} 0 & \phi_0 \\ t 1_A & 0 \end{pmatrix}.$$
In Propositions \ref{key}, \ref{propformdef}, we show that the curvature $(\eta_0)_{A_t}$ of $A_t$ vanishes. The conditions in Theorem \ref{thmintro} further ensure that we can make $B = \Hom(A_t, A_t)$ into a Morita deformation of $A_0$, which is in the pre-image of $\phi$ under $\Theta$. The vanishing of $(\eta_0)_{A_t}$ heavily depends upon regularity of $R = k[[t]]$. 

Let us now consider the infinitesimal setup $R = k[t]/t^{n+1}$. For $\phi = \sum_{k = 0}^{n-1} \phi^{(k)}t^k \in \Sigma \CC(A_0)[t]/t^n$, we consider $\phi t = \sum_{k = 0}^{n-1} \phi^{(k)}t^{k+1} \in \Sigma \CC(A_0)[t]/t^{n+1}$.
 Put 
$$\mathrm{MC}(A_0; k[t]/t^{n+1}) = \{ \phi \in (\Sigma \CC(A_0)[t]/t^n)^1 \,\,| \,\, [\mu^{(0)}, \phi t] + \frac{1}{2}[\phi t,\phi t] = 0 \}/ \sim,$$
the set of elements $\phi$ of shifted degree $1$ for which $\phi t \in \Sigma \CC(A_0)[t]/t^{n+1}$ is a solution of the Maurer Cartan equation up to gauge equivalence of the elements $\phi t$. Let $\Def_{\mathrm{Mor}}(A_0; k[t]/t^{n+1})$ be the set of $n$-th order Morita deformations of $A_0$ up to Morita equivalence.
We obtain a map
\begin{equation}
\Theta_n: \Def_{\mathrm{Mor}}(A_0; k[t]/t^{n+1}) \lra \mathrm{MC}(A_0; k[t]/t^{n+1}).
\end{equation}
It was already noted in \cite{kellerlowen} that in general, the analogue of Theorem \ref{thmintro} does not hold for $\Theta_n$ (see Example \ref{example}). However, in this paper we prove a more refined result, stated more generally for dg categories in Theorem \ref{thmmaininf}. The theorem essentially says that $n+1$-th order curved deformations give rise to $n$-th order uncurved Morita deformations, and explains how the formal object $A_t$ defined earlier on in fact arises as a limit of infinitesimal objects:

\begin{theorem}\label{thmintro2}
Let $A_0$ be a dg $k$-algebra which is cohomologically bounded above.
\begin{enumerate}
\item The map $\Theta_n$ is injective.
\item If $\phi  \in \mathrm{MC}(\Sigma \CC(\AAA_0)[t]/t^{n+1})$ is in the image of $\mathrm{MC}(\Sigma \CC(\AAA_0)[t]/t^{n+2})$, then $\phi$ is in the image of $\Theta_n$.
\end{enumerate}
\end{theorem}

In other words, if the $n$-th order $cA_{\infty}$-deformation $\mu = \mu^{(0)} + \phi t$ of $\mu^{(0)}$ extends to an $n+1$-th order deformation $\bar{\mu} = \mu^{(0)} + \phi t + \phi^{(n+1)}t^{n+1}$, then the deformation $\mu$ can be replaced by an $n$-th order Morita deformation $B$ of $A_0$. 
As before, $B$ is obtained as $B = \Hom(A_{(n)}, A_{(n)})$ for an object $A_{(n)} \in \Tw(A, \mu)$ for $A = A_0[t]/t^{n+1}$. This time, we have to take $A_{(n)} = A \oplus \Sigma^{-1}A$ with 
$$\delta_{A_{(n)}} = \begin{pmatrix} 0 & \phi_0 + \phi^{(n+1)}_0 t^n \\ t 1_A & 0 \end{pmatrix},$$
whose curvature vanishes by Proposition \ref{propinfdef}.
Consequently, the twisted object $A_t$ which is constructed in the formal setup can be interpreted as a $t$-adic limit of the objects $A_{(n)}$, and similarly for the corresponding (torsion) Morita deformations. 

Finally, let us comment upon the relation with other work. 
Although the goal in this paper is independent of the question as to the definition of suitable alternative ``derived categories'' in various contexts, our work is indirectly related to this question. The torsion derived category which we use in \S \ref{parpartor} seems closely related to the derived category of a formal deformation of an algebroid stack over a smooth complex algebraic variety, which is considered and proven to be compactly generated by Petit in \cite{petit}. Futher, for a $k[[t]]$-linear $cA_{\infty}$-algebra $A$ as before, one may expect that a good alternative derived category should be such that the object $A_t$ we construct constitutes a compact generator of this category. This seems to be in accordance with the work of Positselski in \cite{positselski3}, although details remain to be elucidated. Eventually, we expect our setup to be naturally part of a richer picture including a suitable version of MGM duality \cite{alonsojeremiaslipman}, \cite{kashiwarashapira}, \cite{portashaulyekutieli}, \cite{PSY2}, \cite{positselski3}, \cite{positselski:contra}, \cite{positselski:MGM}.

Throughout the paper, little attention is given to the signs in formulas. Picking a consistent sign convention and viewing elements as living in appropriate shifts of complexes like the Hochschild complex, the reader will be able to make the signs more precise where necessary (see for instance \cite[\S 2.1]{lowen9}).

\vspace{0,5cm}

\noindent \emph{Acknowledgement.} 
The authors are grateful to Anthony Blanc, Jim Stasheff and Amnon Yekutieli for their interesting comments and help with references and typo's based upon a previous version of the paper.

\section{Curved $A_{\infty}$-categories and twisted objects}\label{parpartwist}

Let $R$ be a commutative ground ring, endowed with an ideal $T \subseteq R$ and the corresponding $T$-adic topology.
In this section, we introduce a certain construction of uncurved twisted objects over an $R$-linear $cA_{\infty}$-category $\AAA$. In \S \ref{partwist}, we introduce the $cA_{\infty}$-category $\Tw(\AAA)$ of twisted objects over $\AAA$. An object of $\Tw(\AAA)$ is a formal expression $M = \oplus_{i \in I} \Sigma^{n_i} A_i$ with $A_i \in \AAA$ and $n_i \in \Z$, endowed with a connection $\delta_M \in \Hom^1(M,M)$. Here, $\Hom$-modules are defined in the obvious way as modules of ($T$-adic limits of) infinite column finite matrices. The $cA_{\infty}$-structure on $\Tw(\AAA)$ is given by
\begin{equation}\label{eqembrpar}
\embr_{\delta}(\mu) = \sum_{n = 0}^{\infty} \mu\{ \delta^{\otimes n}\} = \mu + \mu\{\delta\} + \mu\{\delta, \delta\} + \dots
\end{equation}
for the brace algebra structure on the shifted Hochschild object $\Sigma \CC(\Tw(\AAA))$. Here, the notation $\embr_{\delta}(\mu)$ stands for ``$\mu$ embracing $\delta$''.
Obviously, we have to make sure that \eqref{eqembrpar} converges in an appropriate way. For our purpose, we consider ``pointwise $T$-adic'' convergence in \S \ref{parcompl}. In order to ensure this convergence, we require connections to be \emph{locally $\mu$-allowable} in the sense of Definition \ref{defnilp}. Even for $T = 0$, this condition is weaker than the usual upper triangular condition used in the definition of classical (uncurved) twisted objects over a dg or $A_{\infty}$-category, as discussed in Example \ref{example}.

Consider $M,N \in \Tw(\AAA)$, $f \in \Hom^1(M,N)$ and $g \in \Hom^1(N,M))$. In \S \ref{partwocone}, we discuss the possiblility of constructing a ``two-sided cone'' $\cone(f,g) = M \oplus N$ endowed with the connection
$$\delta_{\cone(f,g)} = \begin{pmatrix} \delta_M & f \\ g & \delta_N \end{pmatrix}$$
in $\Tw(\AAA)$.
In general, the construction of the two-sided cone faces convergence issues, but in \S \ref{parremcurv} it is used in an effective way to remove curvature. Precisely, if the curvature $c_M$ of an object $M \in \Tw(\AAA)$ can be written as $c_M = r \cdot \frac{c_M}{r}$ for a regular element $r \in R$, then we show that $\cone(\frac{c_M}{r}, r)$ is uncurved (Proposition \ref{key}). In \S \ref{parclosed}, a related construction is given which allows to construct closed elements (Proposition \ref{propcone}).

\subsection{Curved $A_{\infty}$-categories and twisted objects}\label{partwist}
In this section we briefly recall the notion of twisted objects over a $cA_{\infty}$-category. This notion generalizes the classical notion for dg categories from \cite{bondalkapranov}, which was generalized to the setup of $A_{\infty}$-categories in \cite{lefevre}, to the setup of $cA_{\infty}$-categories. For a more detailed treatment we refer to \cite{dedekenlowen2} and the references therein.

Let $R$ be a commutative ground ring. Throughout, $R$-modules will be $\Z$-graded $R$-modules.
Let $\AAA$ be an $R$-quiver, i.e. $\AAA$ consists of a set of objects $\Ob(\AAA)$ and for all $A, A' \in \Ob(\AAA)$, an $R$-module $\AAA(A,A')$.
The Hochschild object of $\AAA$ is the product total  $R$-module $\CC(\AAA)$ of the double $R$-module
$$\CC^{p,q}(\AAA) = \prod_{A_0, \dots, A_p \in \AAA} \Hom^q_R(\AAA(A_{p-1}, A_p) \otimes_R \dots \otimes_R \AAA(A_0, A_1), \AAA(A_0, A_p)).$$
In particular, we have
$$\CC^{0,\ast} = \prod_{A \in \AAA} \AAA(A,A).$$
An element $\phi \in \CC^n(\AAA)$ is given by  $\phi = (\phi_p)_{p \geq 0}$ with $\phi_p \in \CC^{p, n-p}(\AAA)$.
The $R$-module $\CC(\AAA)$ is endowed with brace operations
$$\begin{aligned} \CC^{p, \ast}(\AAA) \otimes \bigotimes_{i = 1}^k \CC^{p_i, \ast}(\AAA) & \lra \CC^{p + p_1 + \dots + p_k - k, \ast}(\AAA) \\
(\phi, \phi_1, \dots, \phi_k) & \longmapsto \phi \{ \phi_1, \dots \phi_k \}
\end{aligned}$$
with
$$\phi \{ \phi_1, \dots \phi_k \} = \sum \phi(1 \otimes \dots \otimes \phi_1 \otimes \dots \otimes \phi_i \otimes 1 \otimes \dots \otimes \phi_n \otimes \dots).$$
The first brace operation is the dot product
$$\begin{aligned} \CC^{n, \ast}(\AAA) \otimes \CC^{m, \ast}(\AAA) & \lra \CC^{n + m -1, \ast}(\AAA) \\
(\phi, \psi) & \longmapsto\phi \bullet \psi = \phi\{ \psi\} = \phi(1 \otimes \psi) + \phi(\psi \otimes 1).
\end{aligned}$$
The brace operations turn $\CC(\AAA)$ into a brace algebra, see \cite{lowen9}.

A \emph{$cA_{\infty}$-structure} on $\AAA$ is an element $\mu = (\mu_n)_{n \geq 0} \in \CC^2(\AAA)$ with $\mu \bullet \mu = 0$.
We call $(\AAA, \mu)$ a \emph{$cA_{\infty}$-category}. The \emph{curvature} of $A \in \AAA$ is the element $$c_A = (\mu_0)_A \in \AAA(A,A)^2.$$
An element $f \in \AAA(A, A')$ is called \emph{closed} if $\mu_1(f) = 0$.

Let $\AAA$ be an $R$-quiver. A \emph{free object} over $\AAA$ is a formal expression $M = \oplus_{i \in I} \Sigma^{m_i} A_i$ with $I$ an arbitrary index set, $A_i \in \AAA$ and $m_i \in \Z$. For another free object $N = \oplus_{j \in J} \Sigma^{n_j} B_j$ over $\AAA$, we put
$$\Hom(M, N) = \prod_{i} \oplus_j \Sigma^{n_j - m_i}\AAA(A_i, B_j).$$
An element $f \in \Hom^l(M,N)$ is represented by a column finite 
matrix $f = (f_{ji})$ with $f_{ji} \in \AAA^{l +n_j - m_i}(A_i, B_j)$. We denote $f(i) = \sum_{j \in I} f_{ji} \in \oplus_j \AAA^{l +n_j - m_i}(A_i, B_j).$
Free objects over $\AAA$ constitute an $R$-quiver $\Free(\AAA)$ with $\Free(\AAA)(M,N) = \Hom(M, N)$. 

\subsection{Completions}\label{parcompl}
Let $R$ be a ring with an ideal $T \subseteq R$, and let $R$ be complete for the $T$-adic topology. Sometimes it is natural to work with completions with respect to this topology, or pointwise versions of this topology. We will make a part, and as we are dealing with completion we are actually using underlying uniform structures \cite{bourbaki2}. Let $\AAA$ be a complete $R$-quiver. For $M = \oplus_{i \in I} \Sigma^{m_i} A_i$ and $N = \oplus_{i \in I} \Sigma^{n_i} A_i$, we endow $\Free(\AAA)(M,N)$ with the pointwise $T$-adic uniformity, for which a sequence $(f_n)_n$ with $f_n \in \Free_{}(\AAA)(M,N)$ is Cauchy if and only if for every $i \in I$, for every $\alpha \in \N$, there is an $n_0 \in \N$ such that for all $p,q \geq n_0$ we have $f_p(i) - f_q(i) \in T^{\alpha}[\oplus_j \Sigma^{n_j - m_i}\AAA(A_i, B_j)]$. Let $\widehat{\Hom}(M,N) = \widehat{\Free}(\AAA)(M,N)$ be the corresponding completion. We thus obtain the quiver $\widehat{\Free}(\AAA)$.

\begin{remark}
Note that an even weaker notion of convergence is obtained by allowing convergence within every column $\oplus_j \AAA^{l +n_j - m_i}(A_i, B_j)$, considered as a submodule of $\prod_j \AAA^{l +n_j - m_i}(A_i, B_j)$, to be pointwise as well. Pointwise convergence for direct sum modules naturally leads to completions of ``decaying functions'' as considered in \cite{yekutieli4}. 
\end{remark}

Now consider the natural extension of an element $\phi \in \CC(\AAA)$ to $\CC(\Free_{}(\AAA))$. Consider 
$\phi: \Free_{}(\AAA)^{\otimes k} \lra \Free_{}(\AAA)$.
By $R$-multilinearity, $\phi$ is continuous in each variable, and we obtain a unique extension
$$\hat{\phi}: \widehat{\Free}_{}(\AAA)^{\otimes k} \lra \widehat{\Free}_{}(\AAA).$$

Now consider a sequence $(\phi_n)_n$ with $\phi_n \in \CC^k(\Free_{}(\AAA))$. We endow $\CC^k(\Free_{}(\AAA))$ with the pointwise uniformity of the codomain $\Free_{}(\AAA)$, for which a sequence $(\phi_n)_n$ is Cauchy if and only if $(\phi_n(f_k, \dots, f_1))_n$ is Cauchy in $\Free_{}(\AAA)$ for every $(f_k, \dots, f_1) \in \Free_{}(\AAA)^{\otimes k}$. We endow $\CC(\Free_{}(\AAA))$ with the product uniformity of the $\CC^k(\Free_{}(\AAA))$. We obtain the completions $\widehat{\CC}^k(\Free_{}(\AAA))$ and $\widehat{\CC}(\Free_{}(\AAA))$.
We have
$$\begin{aligned}
\widehat{\CC}^k(\Free_{}(\AAA)) & = \Hom_R(\Free_{}(\AAA)^{\otimes k} , \widehat{\Free}_{}(\AAA)) \\
& =  \Hom_R(\widehat{\Free}_{}(\AAA)^{\otimes k} , \widehat{\Free}_{}(\AAA)) \\
& = \CC^k(\widehat{\Free}_{}(\AAA))
\end{aligned}$$
and consequently 
\begin{equation}\label{compl}
\widehat{\CC}(\Free_{}(\AAA)) = \CC(\widehat{\Free}_{}(\AAA)).
\end{equation}
In particular, $\widehat{\CC}(\Free_{}(\AAA))$ is a brace algebra

On the other hand, we can endow $\Free_{}(\AAA)$ with the $T$-adic topology, and consider the completion ${\Free}^T_{}(\AAA)$, or we can endow $\CC(\Free_{}(\AAA))$ with the $T$-adic topology, and consider the completion $\CC^T(\Free_{}(\AAA))$. As sets, these completions are equal to the ones we have constructed, but our pointwise uniformities allow a weaker convergence.

\begin{lemma}\label{lemcompl}
With notations as above, the canonical continuous maps ${\Free}^T_{}(\AAA) \lra \widehat{\Free}_{}(\AAA)$ and $\CC^T(\Free_{}(\AAA)) \lra \widehat{\CC}(\Free_{}(\AAA))$ are bijective.
\end{lemma}

\begin{example}
Take $0 = T \subseteq R$. A cauchy sequence $(f_n)_n$ in $\Free_{}(\AAA)(M,N)$ is columnwise stationary. Hence, it converges pointwise to the column-finite matrix of stable columns. Thus, $\Free_{}(\AAA)$ is complete and by \eqref{compl}, so is $\CC(\Free_{}(\AAA))$.
\end{example}

\begin{example}
Take $R = k[[t]]$ and $T = (t) \subseteq R$. Let $\AAA_0$ be a $k$-linear quiver and put $\AAA = \AAA_0[[t]]$. We have $\Free_{}(\AAA) = \Free_{}(\AAA_0)[[t]]$ and $\widehat{\CC}(\Free_{}(\AAA)) = \CC(\Free_{}(\AAA_0))[[t]]$.
\end{example}

\begin{definition}\label{defallow}
Let $\AAA$ be a quiver and consider $\phi, \delta, \psi  \in \widehat{\CC}(\Free_{}(\AAA))$. 
\begin{enumerate}
\item We say that $\delta$ is $\phi$-\emph{allowable} if the series
$\sum_{n \in \N} \phi\{ \delta^{\otimes n} \}$ is convergent. In this case we denote the limit by $\embr_{\delta}(\phi)$.
\item We say that $(\delta, \psi)$ is $\phi$-\emph{allowable} if the series
$\sum_{n \in \N} \sum_{p = 0}^n \phi\{ \delta^{\otimes p} \}\{ \psi^{\otimes n - p}\}$ is convergent (after renumbering, i.e without brackets).
\end{enumerate}
\end{definition}

\begin{remark}
As explained in \cite[\S 3.2]{lowen9}, allowability can be equivalently expressed in terms of convergence of ``$\sum_{n \in \N} \delta^{\otimes n}$'' for an induced ``weak'' topology on the bar construction of the Hochschild complex, and the notion is inspired by the treatment in \cite{fukaya}. It also bears some resemblance to the definition of complete bar resolution from \cite{yekutieli}. 
\end{remark}

\begin{lemma}\label{lemplus}
Consider $\phi \in \widehat{\CC}(\Free_{}(\AAA))$ and $\delta, \psi  \in \widehat{\CC}^{0, \ast}(\Free_{}(\AAA))$ and suppose $(\delta, \psi)$ is $\phi$-allowable. 
\begin{enumerate}
\item $(\psi, \delta)$, $\delta$, $\psi$ and $\delta + \psi$ are $\phi$-allowable.
\item $\psi$ is $\embr_{\delta}(\phi)$-allowable and $\embr_{\psi}(\embr_{\delta}(\phi)) = \embr_{\delta + \psi}(\phi)$.
\end{enumerate}
\end{lemma}

\begin{proof}
(1)  $\phi$-allowability of $\delta$ and $\psi$ readily follows. Since $\delta, \psi  \in \CC^{0, \ast}(\Free_{}(\AAA))$, we have $\phi\{ (\delta + \psi)^{\otimes n} \} = \sum_{p = 0}^n \phi\{\delta^{\otimes p} \} \{ \psi^{\otimes{n-p}} \}$, whence $\phi$-allowablility of $\delta + \psi$ also follows. 
(2) We have $\sum_{n} \embr_{\delta}(\phi) \{ \psi^{\otimes k} \} = \sum_n (\sum_k \phi \{ \delta^{\otimes k} \}) \{\psi^{\otimes n} \} = \sum_n \sum_k (\phi \{ \delta^{\otimes k} \} \{\psi^{\otimes n} \})$. Convergence of this series (in $n$) easily follows from $\phi$-allowability of $(\delta, \psi)$.
\end{proof}

We are mainly interested in the case where $\phi \in \CC(\AAA)$ and $\delta, \psi  \in \CC^{0, \ast}(\Free_{}(\AAA))$. In this setup, we have the following building blocks of allowable elements.

\begin{definition}\label{defnilp}
Consider $\Phi \subseteq \CC(\AAA)$, $M = \oplus_{i \in I} \Sigma^{n_i}A_i  \in \Free_{}(\AAA)$ and $\delta, \psi \in \Hom(M, M)$. We say that 
\begin{enumerate}
\item[(i)] $\delta$ is \emph{locally $\Phi$-allowable}
\item[(ii)] $(\delta, \psi)$ is \emph{locally $\Phi$-allowable}
\end{enumerate}
if for every $\alpha \in \N$ and $i \in I$, there is an $m_0 \in \N$ such that for every $\phi \in \Phi$, for every $j \in I$, for every collection of objects $B_0, \dots, B_n, C_0, \dots, C_l \in \AAA$ with $B_n = A_i$ and $C_0 = A_j$, for every $(b_n, \dots, b_1) \in \AAA(B_{n-1}, B_n) \otimes \dots \otimes \AAA(B_0, B_1)$ and $(c_l, \dots, c_1) \in \AAA(C_{l-1}, C_l) \otimes \dots \otimes \AAA(C_0, C_1)$, we have
\begin{enumerate}
\item[(i)] for every $m \geq m_0$, the expression
\begin{equation}\label{eqnilp}
\sum_{k_{m-1}, \dots, k_1}\phi(c_l \otimes \dots \otimes c_{1} \otimes \delta_{jk_{m-1}} \otimes \dots \delta_{k_{d+1} k_d} \dots \otimes \delta_{k_1 i} \otimes b_{n} \otimes \dots \otimes b_1)
\end{equation}
is in $T^{\alpha}\AAA(B_0, C_l)$.
\item[(ii)] for every $p, q$ with $p + q = m \geq m_0$, the expression
\begin{equation}\label{eqnilp2}
\sum_{k_{m-1}, \dots, k_1}\phi(c_l \otimes \dots \otimes c_{1} \otimes \delta^{s_m}_{jk_{m-1}} \otimes \dots \delta^{s_{d+1}}_{k_{d+1} k_d} \dots \otimes \delta^{s_1}_{k_1 i} \otimes b_{n} \otimes \dots \otimes b_1)
\end{equation}
is in $T^{\alpha}\AAA(B_0, C_l)$, where the sum is taken over all expressions with $p$ of the $\delta^{s_i}$ equal to $\delta$ and $q$ of the $\delta^{s_i}$ equal to $\psi$.
\end{enumerate}
For $\Phi = \{ \phi \}$, we simply say locally $\phi$-nilpotent instead of locally $\Phi$-nilpotent.
\end{definition}

\begin{lemma}\label{lemnilp}
If $(\delta, \phi)$ is locally $\Phi$-nilpotent, then $\delta$, $\psi$ and $\delta + \psi$ are locally $\Phi$-nilpotent.
\end{lemma}

\begin{lemma}
If $\delta$ is locally $\Phi$-allowable, and $\psi \in T^{1}\Hom(M,M)$, then $(\delta, \psi)$ is locally $\Phi$-allowable.
\end{lemma}

\begin{proof}
Take $\alpha \in \N$ and $i \in I$. Since $\delta$ is locally $\Phi$-allowable, we obtain $m_0 \in \N$ as in Definition \ref{defnilp}. Put $m_1 = \alpha m_0$ and consider $p + q \geq m_1$. If $q \geq \alpha$, then expression \eqref{eqnilp2} is in $T^{\alpha} \AAA$ and we are done. If $q < \alpha$, then expression \eqref{eqnilp2} contains at least $m_0$ consecutive entries $\delta^{s_i}$. Hence, by $\Phi$-allowability of $\delta$, expression \eqref{eqnilp2} is in $T^{\alpha} \AAA$ as well.
\end{proof}

\begin{example}\label{exiln}
Put $0 = T \subseteq R$. Consider $M = \oplus_{i \in I} \Sigma^{n_i}A_i  \in \Free(\AAA)$ and $\delta \in \Hom(M, M)$. We say that $\delta$ is \emph{upper triangular} (called ``intrinsically locally nilpotent'' in \cite{dedekenlowen2}) if there is an $m_0 \in \N$ such that every sequence $(\delta_{jk_{m-1}}, \dots, \delta_{k_{d+1} k_d}, \dots, \delta_{k_1 i})$ of length $m \geq m_0$ contains an zero element. Obviously, if $\delta$ is upper triangular, it is locally $\CC(\AAA)$-allowable. If $\delta$ corresponds to an upper triangular $n \times n$ matrix, then $\delta$ is upper triangular by taking $m_0 = n$ for every $i \in I$. 
\end{example}

\begin{example}\label{exmun}
Consider $M = \oplus_{i \in I} \Sigma^{n_i}A_i  \in \Free(\AAA)$ and $\delta, \psi \in \Hom(M, M)$. If for $\Phi \subseteq \CC(\AAA)$, for every $\alpha \in \N$, there exists $n_0 \in \N$ such that for all $\phi \in \Phi$, we have $\phi_n \in T^{\alpha}\CC(\AAA)$ for $n \geq n_0$, then $(\delta, \psi)$ is locally $\Phi$-allowable, by choosing $m_0 = n_0$ for every $i \in I$.
\end{example}

\begin{proposition}\label{propnilpallow}
Consider $\phi \in \CC(\AAA)$, $\delta, \psi \in \CC^{0, \ast}(\Free_{}(\AAA))$. Suppose for every $M \in \Free_{}(\AAA)$, $\delta_M \in \Hom(M, M)$ is locally $\phi$-allowable (resp. $(\delta_M, \psi_M)$ is locally $\phi$-allowable). Then $\delta$ is $\phi$-allowable (resp. $(\delta, \psi)$ is $\phi$-allowable).
\end{proposition}

\begin{proof}
We prove (1). We check convergence of $\sum_{n \in \N} \phi\{ \delta^{\otimes n} \}$. It suffices to look at a collection of objects $M_0, \dots, M_n \in \Free_{}(\AAA)$ with $M_i = \oplus_{k \in K_i} \Sigma^{p_k}B_k$,  $\underline{f} = f_n \otimes \dots \otimes f_1\in \Hom(M_{n-1}, M_n) \otimes \dots \otimes \Hom(M_0, M_1)$ and $k \in K_0$ and $\alpha \in \N$.   We have to show that for some $r_0$, for $r \geq r_0$, $(\phi\{ \delta^{\otimes r} \}(\underline{f}))_{lk} \in T^{\alpha}\AAA$ for all $l \in K_m$. For $\delta_0 = \delta_{M_0} \in \Hom(M_0, M_0)$ and the given $k \in K_0$, let $m_0 = m^{k, \alpha}_0 \in \N$ be obtained from Definition \ref{defnilp}. Let $J_1 \subseteq K_1$ consist of all $l \in K_1$ for which there exists a sequence $(k_m = l, k_{m-1}, \dots, k_0 = k)$ with $m \leq m_0$ such that $((f_1)_{l k_{m-1}}, (\delta_0)_{k_{m-1} k_{m-2}}, \dots, (\delta_0)_{k_1 k})$ contains only non-zero elements. The set $J_1$ is finite. For $\delta_1 = \delta_{M_1} \in \Hom(M_1, M_1)$ and $k \in J_1$, let $m_1^k$ be obtained from Definition \ref{defnilp}. Put $m_1 = \max_{k \in J_1} m_1^k$. Proceeding in the same way, we inductively obtain finite sets $J_i \subseteq K_i$ and $m_i \in \N$ for all $0 \leq i \leq n$. Put $r_0 = \sum_{i = 0}^n m_i$. Now for $r \geq r_0$, $\phi\{ \delta^{\otimes r} \}(\underline{f})$ is a sum of expressions
\begin{equation}\label{eqdelt}
\phi( (\delta_n)^{\otimes t_n}, f_n, \dots, f_{i+1}, (\delta_i)^{\otimes t_i}, f_i, \dots, f_1, (\delta_0)^{\otimes t_0})
\end{equation}
with $\sum_{i = 0}^n t_i = r \geq r_0$ and hence $t_i \geq m_i$ for some $i$. By construction, the $lk$-entry of \eqref{eqdelt} is obtained as a finite sum of expressions of the form \eqref{eqnilp}, taking $\delta = \delta_i$. Hence, this entry is in $T^{\alpha}\AAA$ as desired.
\end{proof}

Let $(\AAA, \mu)$ be a $cA_{\infty}$-category. For an object $M \in \Free(\AAA)$, we consider the set $\Delta_M \subseteq \Hom^1(M,M)$ of locally $\mu$-allowable elements $\delta_M \in \Hom^1(M,M)$, and we consider the completion $\hat{\Delta}_M \subseteq \widehat{\Hom}^1(M,M)$. Elements of $\widehat{\Hom}^1(M,M)$ are called \emph{connections} on $M$ and elements of $\hat{\Delta}_M$ are called \emph{locally $\mu$-allowable connections}.

We now define the quiver ${\Free}_{\hat{\Delta}}(\AAA)$ with 
$$\Ob({\Free}_{\hat{\Delta}}(\AAA)) = \{ (M, \delta_M) \,\, |\,\, M \in \Free(\AAA), \delta_M \in \hat{\Delta}_M \}$$
and ${\Free}_{\hat{\Delta}}(\AAA)((M, \delta_M), (N, \delta_N)) = {\Free}(\AAA)(M,N)$. 
For convenience, we continue to denote objects in ${\Free}_{\hat{\Delta}}(\AAA)$ by $M, N, \dots$. An element $\phi \in \CC(\AAA)$ is trivially extended to an element $\phi \in \CC(\Free_{\hat{\Delta}}(\AAA))$ by mimicking matrix multiplication.
Clearly the results we obtained so far in this section do not depend on the duplication of objects in going from $\Free(\AAA)$ to ${\Free}_{\hat{\Delta}}(\AAA)$, and can readily be stated for ${\Free}_{\hat{\Delta}}(\AAA)$ instead.
Finally, we define the quiver of \emph{twisted objects} over $\AAA$ as
$$\Tw(\AAA) = \widehat{\Free}_{\hat{\Delta}}(\AAA)$$
with $\Ob(\Tw(\AAA)) = \Ob({\Free}_{\hat{\Delta}}(\AAA)$ and $\Tw(\AAA)(M,N) =\widehat{\Free}(\AAA)(M,N)$.
This quiver is endowed with a canonical element $\delta \in \CC^{0,1}(\Tw(\AAA))$ with $\delta_{(M, \delta_M)} = \delta_M$.

In the sequel, we often pretend that connections on $M$ are in $\Hom(M,M)$ rather than in $\widehat{\Hom}(M,M)$. It is easily seen that all the claims we make can actually be extended to the more general connections in $\widehat{\Hom}(M,M)$ by continuity. Further, the connections which will be most important to uss later on, are connections in $\Hom(M,M)$.

\begin{proposition}
There is a $cA_{\infty}$-structure on $\Tw(\AAA)$ given by
\begin{equation}\label{eqembr}
\mathrm{embr}_{\delta}(\mu) = \mu + \mu\{\delta\} + \mu\{ \delta, \delta\} + \dots + \mu\{\delta^{\otimes n} \} + \dots
\end{equation}
For $(M, \delta_M) \in \Tw(\AAA)$, the curvature of $M$ is the element
\begin{equation}
c_M = (\mathrm{embr}_{\delta}(\mu)_0)_M  =  (\mu_0)_M + \mu_1(\delta_M) + \mu_2(\delta_M, \delta_M) + \dots + \mu_n(\delta_M^{\otimes n}) + \dots
\end{equation}
\end{proposition}

\begin{proof}
Convergence of $\embr_{\delta}(\mu)$ follows from Proposition \ref{propnilpallow}. A computation reveals that $\mathrm{embr}_{\delta}(\mu) \bullet \mathrm{embr}_{\delta}(\mu) = \mathrm{embr}_{\delta}(\mu \bullet \mu)$ whence $\embr_{\delta}(\mu)$ defines a $cA_{\infty}$-structure on $\Tw(\AAA)$.
\end{proof}

Consider a small full subcategory $\BBB \subseteq \Tw(\AAA)$ (endowed with the induced $cA_{\infty}$-structure). There is a natural fully faithful morphism of quivers
$$\varphi: \Free(\BBB) \lra \Free(\AAA): \oplus_{i \in I} \Sigma^{n_i} (\oplus_{j \in J_i} \Sigma^{m_j} A_{ij}) \longmapsto \oplus_{i \in I, j \in J_i} \Sigma^{n_i + m_j} A_{ij}.$$ 
Consider a collection of objects $(B_i = \oplus_{j \in J_i} \Sigma^{m_j} A_{ij}, \delta_{B_i}) \in \Tw(\AAA)$. For $N = \oplus_{i \in I} \Sigma^{n_i} B_i \in \Free(\BBB)$, we can
endow the image $\varphi(N)$ with the block diagonal connection
$$\delta_{\oplus} = \begin{pmatrix} {\dots} & 0 & 0 \\ 0 & (-1)^{n_i} \delta_{B_i} & 0\\ 0 & 0 & {\dots} \end{pmatrix}$$
for which we have $(\varphi(N), \delta_{\oplus}) \in \Tw(\AAA)$. 

\begin{example}\label{exshiftsum}
Every object $M = \oplus_{j \in J} \Sigma^{m_j} A_j$ in $\Tw(\AAA)$ has an $n$-shift $\Sigma^n M = \oplus_{j \in J} \Sigma^{m_j + n} A_j$ in $\Tw(\AAA)$ and every collection of objects $M_i = \oplus_{j \in J_i} \Sigma^{m_j} A_{ij} \in \Tw(\AAA)$ has a \emph{direct sum} $\oplus_{i \in I} \oplus_{j \in J_i} \Sigma^{m_j} A_{ij} \in \Tw(\AAA)$.
\end{example}

Now suppose we have $(N = \oplus_{i \in I} \Sigma^{n_i} B_i, \psi_N) \in \Tw(\BBB)$. Then we obtain another connection 
$$\varphi(\psi_N) \in \Hom^1(\varphi(N), \varphi(N)).$$ 

Let $\Tw'(\BBB) \subseteq \Tw(\BBB)$ be the full $cA_{\infty}$-subcategory of twisted objects $(N, \psi_N)$ for which $(\delta_{\oplus}, \varphi(\psi_N))$ is locally $\mu$-allowable.

\begin{proposition}\label{propcomp}
The natural morphism of quivers
$$\varphi: \Tw'(\BBB) \lra \Tw(\AAA): (N = \oplus_{i \in I} \Sigma^{n_i} B_i, \psi_N) \longmapsto (\varphi(N), \delta_{\oplus} + \varphi(\psi_N))$$
is an embedding of $\Tw'(\BBB)$ as a fully faithful $cA_{\infty}$-subcategory of $\Tw(\AAA)$ (with induced $cA_{\infty}$-structure).
\end{proposition}

\begin{proof}
If $(\delta_{\oplus}, \varphi(\psi_N))$ is locally $\mu$-allowable, we have $(\varphi(N), \delta_{\oplus} + \varphi(\psi_N)) \in \Tw(\AAA)$. 
Denoting connections in $\Tw(\AAA)$ by $\delta$, and connections in $\Tw(\BBB)$ by $\psi$, the $cA_{\infty}$-structure on (subcategories of) $\Tw(\AAA)$ is given by $\embr_{\delta}(\mu)$ and the $cA_{\infty}$-structure on (subcategories of) $\Tw(\BBB)$ is given by $\embr_{\psi}(\embr_{\delta}(\mu))$. Here $\embr_{\delta}(\mu)$ is to be interpreted as the natural extension of $\embr_{\delta}(\mu)$ on $\BBB$ to $\Tw(\BBB)$. Using $\varphi$, this structure translates into $\embr_{\delta_{\oplus}}(\mu)$, and $\embr_{\psi}(\embr_{\delta}(\mu))$ translates into $$\embr_{\varphi(\psi)}(\embr_{\delta_{\oplus}}(\mu)).$$ By Lemma \ref{lemplus}, the latter structure equals $\embr_{\delta_{\oplus} + \varphi(\psi)}(\mu)$, which is the restriction of $\embr_{\delta}(\mu)$ to the image of $\varphi$.
\end{proof}

\begin{example}\label{excone}
Consider $(M, \delta_M), (N, \delta_N) \in \Tw(\AAA)$ and $f \in \Hom^1(M,N)$. On $M \oplus N$, consider the connections
$$\delta_{\oplus} = \begin{pmatrix} \delta_M & 0 \\ 0 & \delta_N \end{pmatrix}; \hspace{1cm} \delta_f = \begin{pmatrix} 0 & f \\ 0 & 0 \end{pmatrix}.$$
The couple $(\delta_{\oplus}, \delta_f)$ is localy $\mu$-nilpotent. Indeed, for given $i \in I$ (assocated to $M \oplus N$) and $\alpha \in \N$, let $m_0$ be such that for $m \geq m_0$, expression \eqref{eqnilp} is in $T^{\alpha}\AAA$ for $\delta_{\oplus}$. Take $m_0' = m_0 + 2$. As soon as a term in expression \eqref{eqnilp2} contains at least 2 factors coming from $\delta_f$, the term vanishes. 
We thus obtain $\cone(f) = M \oplus N$ endowed with the connection
$$\delta_{\cone(f)} = \begin{pmatrix} \delta_M & f \\ 0 & \delta_N \end{pmatrix}$$
in $\Tw(\AAA)$.
Put $\BBB = \{ M, N\} \subseteq \Tw(\AAA)$, endowed with $cA_{\infty}$-structure $\eta = \embr_{\delta}(\mu)$. In $\Tw(\BBB)$, we consider $\cone_{\BBB}(f) = M \oplus N$ endowed with connection
$\psi_f =  \begin{pmatrix} 0 & f \\ 0 & 0 \end{pmatrix}$. We have $\varphi(\cone_{\BBB}(f)) = \cone(f)$. The curvature of $\cone_{\BBB}(f)$ is
$$c = (\eta_0)_{M \oplus N} + \eta_1(\psi_f) = \begin{pmatrix} (\eta_0)_M & \eta_1(f) \\ 0 & (\eta_0)_N \end{pmatrix}.$$
According to Proposition \ref{propcomp}, the curvature of $\cone(f) \in \Tw(\AAA)$ is 
$$\varphi(c) = \begin{pmatrix} c_M & \embr_{\delta}(\mu)(f) \\ 0 & c_N \end{pmatrix}.$$
\end{example}

\begin{remark}
Let $\Tw(\AAA)_{\infty} \subseteq \Tw(\AAA)$ be the full subcategory of \emph{uncurved} objects, i.e. objects $M$ with $c_M = 0$. By Examples \ref{exshiftsum} and \ref{excone}, the category $H^0(\Tw(\AAA)_{\infty})$ naturally becomes a triangulated category. Suitable subcategories of $\Tw(\AAA)_{\infty}$ can be used to model familiar triangulated categories. For instance, for an $A_{\infty}$-category $\AAA$ and $T = 0$, the full subcategory $\Tw_{ut}(\AAA)_{\infty}$ of uncurved upper triangular objects in the sense of Example \ref{exiln} models the derived category of $\AAA$, that is, $H^0(\Tw_{ut}(\AAA)_{\infty}) \cong D(\AAA)$ (see also \cite{lowen9}, \cite{dedekenlowen2}).
\end{remark}

\subsection{Two-sided cones} \label{partwocone}

Let $(\AAA, \mu)$ be a $cA_{\infty}$-category.
In this section, we discuss a cone-like construction in $\Tw(\AAA)$. A similar construction was considered in \cite{kellerlowennicolas} in the context of cdg algebras and modules.
Consider $(M, \delta_M), (N, \delta_N) \in \Tw(\AAA)$ and $f \in \Hom^1(M,N)$, $g \in \Hom^1(N,M)$. On $M \oplus N$, consider the connections
$$\delta_{\oplus} = \begin{pmatrix} \delta_M & 0 \\ 0 & \delta_N \end{pmatrix}; \hspace{1cm} \delta_{f,g} = \begin{pmatrix} 0 & f \\ g & 0 \end{pmatrix}.$$
If $(\delta_{\oplus}, \delta_{f,g})$ is locally $\mu$-nilpotent, we define the twisted object $\cone(f,g) = M \oplus N$ endowed with the connection
$$\delta_{\cone(f,g)} = \begin{pmatrix} \delta_M & f \\ g & \delta_N \end{pmatrix}$$
in $\Tw(\AAA)$.

Let us investigate when $(\delta_{\oplus}, \delta_{f,g})$ is locally $\mu$-nilpotent. In order to evaluate expression \eqref{eqnilp2}, we look at sequences
\begin{equation}\label{eqseq}
(\delta_{\oplus}^{\otimes n_l} , \delta_{f,g}, \delta_{\oplus}^{\otimes n_{l -1}}, \delta_{f,g}, \dots, \delta_{f,g}, \delta_{\oplus}^{\otimes n_1}),
\end{equation}
where $\delta_{\oplus}^{\otimes k}$ stands for $k$ consecutive entries $\delta_{\oplus}$ in the sequence.
Letting $M$ correspond to $0$ and $N$ to $1$, we systematically look at the four blocks $00$, $01$, $10$ and $11$ for elements in $\Hom(M \oplus N, M \oplus N)$. 
For a sequence $(s_p, s_{p-1}, \dots, s_1)$ like \eqref{eqseq}, we are interested in sequences
$((s_p)_{k_p k_{p-1}}, (s_{p-1})_{k_{p-1}, k_{p-2}}, \dots, (s_1)_{k_1, k_0})$ for $k_i \in \{0, 1\}$, which do not contain zero blocks. For \eqref{eqseq}, the only such sequences are
\begin{equation}\label{eqkey}
(\delta_M^{\otimes n_l}, f, \delta_N^{\otimes n_{l-1}}, g, \delta_M^{\otimes n_{l-2}}, \dots, g, \delta_M^{\otimes n_{1}})
\end{equation}
and the variants starting and or ending with a power of $\delta_N$.

\begin{proposition}\label{propfg}
Suppose for $\alpha \in \N$, there is an $m_0 \in \N$ such that $\mu_p(h_p, \dots, h_1) \in T^{\alpha}\AAA$ as soon as there are $m_0$ different indices $i$ for which $h_i = g$. Then $(\delta_{\oplus}, \delta_{f,g})$ is locally $\mu$-allowable. In particular, if $g \in T^1\AAA$, then $(\delta_{\oplus}, \delta_{f,g})$ is locally $\mu$-allowable.
\end{proposition}

\begin{proof}
Using \eqref{eqkey}, his is proven in a similar way as Proposition \ref{propnilpallow}.
\end{proof}
Suppose $(\delta_{\oplus}, \delta_{f,g})$ is locally $\mu$-allowable. 
Put $\BBB = \{ M, N\} \subseteq \Tw(\AAA)$, endowed with $cA_{\infty}$-structure $\eta = \embr_{\delta}(\mu)$. In $\Tw(\BBB)$, we consider $\cone_{\BBB}(f,g) = M \oplus N$ endowed with connection
$\psi_{f,g} =  \begin{pmatrix} 0 & f \\ g & 0 \end{pmatrix}$. We have $\varphi(\cone_{\BBB}(f,g)) = \cone(f,g)$. The curvature of $\cone_{\BBB}(f,g)$ is $c = \sum_{n \in \N} \eta_n( \psi_{f,g}^{\otimes n} )$. We have
\begin{equation}\label{eqeven}
\eta_n(\psi^{\otimes n}_{f,g}) = \begin{pmatrix} \eta_n(f,g,\dots,g) & 0 \\ 0 & \eta_n(g,f,\dots,f) \end{pmatrix}  \hspace{1cm} n \,\, \text{even,}
\end{equation}
\begin{equation}\label{eqodd}
\eta_n(\psi^{\otimes n}_{f,g}) = \begin{pmatrix} 0 & \eta_n(f,g, \dots,f) \\ \eta_n(g,f, \dots,g) & 0 \end{pmatrix} \hspace{1cm} n \,\, \text{odd}.
\end{equation}
In particular, 
\begin{equation}\label{eq00}
\eta_0(\psi^{\otimes 0}_{f,g}) = (\eta_0)_{M \oplus N} = \begin{pmatrix} (\eta_0)_M & 0 \\ 0 & (\eta_0)_N \end{pmatrix}.
\end{equation}
We thus obtain:
$$\begin{aligned}
c & = \begin{pmatrix} (\eta_0)_M & 0 \\ 0 & (\eta_0)_N \end{pmatrix} + \begin{pmatrix} 0 & \eta_1(f) \\ \eta_1(g) & 0 \end{pmatrix} + \begin{pmatrix} \eta_2(f,g) & 0 \\ 0 & \eta_2(g,f) \end{pmatrix}\\
& + \begin{pmatrix} 0 & \eta_3(f,g,f) \\ \eta_3(g,f,g) & 0 \end{pmatrix} + \begin{pmatrix} \eta_4(f,g,f,g) & 0 \\ 0 & \eta_4(g,f,g,f) \end{pmatrix} + \dots\\
& = \begin{pmatrix} (\eta_0)_M + \eta_2(f,g) + \eta_4(f,g,f,g) + \dots & \eta_1(f) + \eta_3(f,g,f) + \dots \\ \eta_1(g) + \eta_3(g,f,g) + \dots  & (\eta_0)_N + \eta_2(g,f) + \eta_4(g,f,g,f) + \dots \end{pmatrix}.
\end{aligned}$$
According to Proposition \ref{propcomp}, the curvature of $\cone(f) \in \Tw(\AAA)$ is $\varphi(c)$. For $g = 0$, we recover Example \ref{excone}.

The condition imposed in Proposition \ref{propfg} in order to obtain local $\mu$-allowablity of $(\delta_{\oplus}, \delta_{f,g})$ is quite restrictive in general. For instance in case $M = N$, then for $g \in \Hom^1(M,M)$, it is not sufficient that $g$ is upper triangular in order to fulfil this property. On the other hand, if $\mu$ is such that for certain $n_0 \in \N$ we have $\mu_n = 0$ for $n \geq n_0$, then every $g \in \Hom^1(M, N)$ satisfies the condition.

\subsection{Removing curvature}\label{parremcurv}

Next we describe a procedure for building uncurved objects starting from curved objects, based upon \S \ref{partwocone}. 

Let $(\AAA, \mu)$ be a small $R$-linear $cA_{\infty}$-category. A \emph{strict unit} for $A \in \AAA$ is an element $1_A \in \AAA^0(A,A)$ with
\begin{enumerate}
\item[(U1)] $\mu_1(1_A) = 0$;
\item[(U2)] $\mu_2(1_A, a) = a$ and $\mu_2(1_A, b) = b$;
\item[(Un)] $\mu_n(a_{n-1}, \dots, 1_A, \dots, a_1) = 0$
\end{enumerate}
for all $n \geq 3$ and for all $a, b, a_1, \dots, a_{n-1} \in \AAA$ for which the expressions make sense.
We call $\AAA$ \emph{strictly unital} if every object $A \in \AAA$ has a strict unit $1_A$. 
Consider $(\Tw(\AAA), \eta = \embr_{\delta}(\mu))$. If $\AAA$ is strictly unital, then so is $\Tw(\AAA)$ and the strict unit $1_M$ for $(M = \oplus_{i \in I} \Sigma^{n_i} A_i, \delta_M)$ corresponds to the diagonal matrix with $(1_M)_{ii} = 1_{A_i} \in \AAA^0(A_i, A_i)$ (see for instance \cite{dedekenlowen2}).

Now suppose $\AAA$ is strictly unital and consider $(M, \delta_M) \in \Tw(\AAA)$.
Suppose $c_M = (\eta_0)_M$ is divisible by a certain element $r \in R$, i.e. there exists an element $\frac{c_M}{r} \in \Hom^2(M,M)$ with $r\frac{c_M}{r} = c_M$. Consider $M_r = M \oplus \Sigma^{-1}M$ and
$$\delta_{M_r} = \begin{pmatrix} \delta_M & \frac{c_M}{r} \\ -r1_M & - \delta_M \end{pmatrix} \in \Hom^1({M}_r, {M}_r).$$

\begin{proposition}\label{key}
\begin{enumerate}
\item On $M_r$, $(\delta_{\oplus}, \delta_{M_r} - \delta_{\oplus})$ is locally $\mu$-nilpotent and we have $(M_r, \delta_{M_r}) = \cone(\frac{c_M}{r},r1_M) \in \Tw(\AAA)$. 
\item We have $$c_{{M}_r} = 0 \iff \eta_1(\frac{c_M}{r}) = 0.$$
\item The condition that $\eta_1(\frac{c_M}{r}) = 0$ is fulfilled if $r$ is regular with respect to the $R$-module $\Hom^3(M,M)$, i.e. if for $m \in \Hom^3(M,M)$ we have that $rm = 0$ implies $m = 0$. In particular it is fulfilled if $r \in R$ is a regular element and $\Hom^3(M,M)$ is a flat $R$-module, and it is always fulfilled for $r = 1 \in R$.
\end{enumerate}
\end{proposition}

\begin{proof}
Clearly, $g = -r1_M$ satisfies the condition in Proposition \ref{propfg} by (Un) for $n \geq 3$ (take $m_0 = 3$). This already shows (1).
Put $\BBB = \{ M \} \subseteq \Tw(\AAA)$ with the $cA_{\infty}$-structure $\eta = \embr_{\delta}(\mu)$. Consider $N = M \oplus \Sigma^{-1}M \in \Tw(\BBB)$ with
$$\psi_{N} = \begin{pmatrix} 0 & \frac{c_M}{r} \\ -r1_M & 0 \end{pmatrix}.$$
Then under the map $\varphi$ from Proposition \ref{propcomp}, we have $\varphi(N, \delta_N) = (M_r, \delta_{M_r})$.  Let $\eta$ denote the $cA_{\infty}$-structure on $\BBB$, with $(\eta_0)_M = c_M$. Put $c_n = \eta_n(\psi_N^{\otimes n})$. By Proposition \ref{propcomp}, it suffices to show for $c = \mathrm{embr}_{\psi_N}(\eta)_0 = \sum_{n \in \N} c_n$ that $c = 0$.
We compute $c$ using \eqref{eqeven} and \eqref{eqodd}. By (Un), $c_n = 0$ for $n \geq 3$.
Further, we have
$$c_0 = \begin{pmatrix} c_M & 0 \\ 0 & c_M \end{pmatrix},$$
using (U1), we have
$$c_1 = \begin{pmatrix} 0 & \eta_1(\frac{c_M}{r}) \\ -r\eta_1(1_M) & 0 \end{pmatrix} = \begin{pmatrix} 0 & \eta_1(\frac{c_M}{r}) \\ 0 & 0 \end{pmatrix}$$
and using (U2) we have
$$c_2 = \begin{pmatrix} -r\eta_2(\frac{c_M}{r},1_M) & 0 \\ 0 & -r\eta_2(1_M, \frac{c_M}{r}) \end{pmatrix} = \begin{pmatrix} -r\frac{c_M}{r} & 0 \\ 0 & -r\frac{c_M}{r}   \end{pmatrix}.$$
By definition of $\frac{c_M}{r}$, we thus have $c_0 + c_2 = 0$ and $c = c_1$, proving (2).

For part (3) of the claim, we first note that
$$\eta_1(c_M) = \eta_1((\eta_0)_M) = ((\eta \bullet \eta)_0)_M = 0$$
since $\eta$ is a $cA_{\infty}$-structure. 
So $$0 = \eta_1(c_M) = \eta_1(r\frac{c_M}{r}) = r\eta_1(\frac{c_M}{r}) \in \Hom^3(M, M).$$ If $r$ is regular with respect to $\Hom^3(M, M)$, it follows that $\eta_1(\frac{c_M}{r}) = 0$, which finishes the proof.
\end{proof}

\subsection{Constructing closed morphisms}\label{parclosed}
The construction in \S \ref{parremcurv} has a parallel story for morphisms of $cA_{\infty}$-category, which will be used later on in \S \ref{parinfty}. Let $\AAA$ be a small, strictly unital $R$-linear $cA_{\infty}$-category and consider $(M, \delta_M)$, $(N, \delta_N) \in (\Tw(\AAA), \eta = \embr_{\delta}(\mu))$ and $f \in \Hom(M,N)$. Suppose $c_M$, $c_N$ and $\eta_1(f)$ are divisible by $r \in R$ and consider the resulting objects $(M_r, \delta_{M_r})$, $(N_r, \delta_{N_r}) \in \Tw(\AAA)$. 
Consider the morphism
$$f_r = \begin{pmatrix} f & \frac{\eta_1(f)}{r} \\ 0 & -f \end{pmatrix} \in \Hom(M_r, N_r).$$

\begin{lemma}\label{lemconee}
There is a canonical isomorphism $\sigma \in \Hom(\cone(f_r), \cone(f)_r)$ given by the permutation of the middle two factors $$\sigma \in \Hom(N \oplus \Sigma^{-1}N \oplus \Sigma M \oplus M, N \oplus \Sigma M \oplus \Sigma^{-1} N \oplus M)$$ fitting into an $\eta_2$-commutative diagram
$$\xymatrix{{M_r = M \oplus \Sigma^{-1} M} \ar[d]_{f_r} \ar[r] & {\Sigma^{-1} M} \ar[d]^{-f} \\
{N_r = N \oplus \Sigma^{-1}N} \ar[r] \ar[d] & {\Sigma^{-1} N} \ar[d] \\ 
{\cone(f_r) \cong_{\sigma} \cone(f)_r} \ar[r] & {\Sigma^{-1} \cone(f)}. }$$
\end{lemma}

\begin{proof}
This is a direct verification.
\end{proof}

\begin{proposition}\label{propcone}
Suppose $r$ is regular with respect to the $R$-modules $\Hom(M,M)$, $\Hom(N,N)$ and $\Hom(M,N)$. Then $c_{M_r} = 0$, $c_{N_r} = 0$, $\eta_1(f_r) = 0$ and $c_{\cone(f)_r} = 0$.
\end{proposition}

\begin{proof}
By proposition \ref{key}, we have $c_{M_r} = 0$, $c_{N_r} = 0$ and by Lemma \ref{lemconee} and Example \ref{excone}, we have
$$c_{\cone(f)_r} \cong c_{\cone(f_r)} = \begin{pmatrix} c_{M_r} & \eta_1(f_r) \\ 0 & c_{N_r} \end{pmatrix}.$$
Hence, it suffices to show that $c_{\cone(f_r)} = 0$. 
Now let $\BBB = \{M, N\} \subseteq \Tw(\AAA)$ be endowed with $\eta = \embr_{\delta}(\mu)$ and consider $M_r, N_r \in \Tw(\BBB)$. Further, let $\BBB_r = \{ M_r, N_r\} \subseteq \Tw(\BBB)$ be endowed with $\embr_{\delta}(\eta)$ and consider $\cone(f_r) \in \Tw(\BBB_r)$.
We thus have
$$c_{\cone(f_r)} = \begin{pmatrix} c_{M_r} & \mathrm{embr}_{\delta}(\eta)_1(f_r) \\ 0 & c_{N_r} \end{pmatrix}.$$
Since $c_{M_r} = 0$ and $c_{N_r} = 0$, it remains to prove that $\mathrm{embr}_{\delta}(\eta)_1(f_r) = 0$.
We calculate
$$\begin{aligned}
\mathrm{embr}_{\delta}(\eta)_1(f_r) & = \eta_1(\begin{pmatrix} f & \frac{\eta_1(f)}{r} \\ 0 & -f \end{pmatrix})\\
& + \eta_2(\begin{pmatrix} 0 & \frac{c_N}{r} \\ -r1_N & 0 \end{pmatrix} , \begin{pmatrix} f & \frac{\eta_1(f)}{r} \\ 0 & -f \end{pmatrix}) + \eta_2(\begin{pmatrix} f & \frac{\eta_1(f)}{r} \\ 0 & -f \end{pmatrix}, \begin{pmatrix} 0 & \frac{c_M}{r} \\ -r1_M & 0 \end{pmatrix}) \\
& + \eta_3(\delta_N, \delta_N, f_r) + \eta_3(\delta_N, f_r, \delta_M) + \eta_3(f_r, \delta_M, \delta_M) + \dots\\
& = \begin{pmatrix} \eta_1(f) & \eta_1(\frac{\eta_1(f)}{r}) \\ 0 & -\eta_1(f) \end{pmatrix}\\
& + \begin{pmatrix} 0 & -\eta_2(\frac{c_N}{r}, f) \\ -rf & -r\frac{\eta_1(f)}{r} \end{pmatrix} + \begin{pmatrix} -r\frac{\eta_1(f)}{r} & \eta_2(f, \frac{c_M}{r}) \\ rf & 0 \end{pmatrix}
\end{aligned}$$
where all terms involving $\eta_k$ with $k \geq 3$ vanish since $\eta$ is strictly unital.
The resulting matrix only features a possibly non-zero upper right entry given by
$$\gamma = \eta_1(\frac{\eta_1(f)}{r}) + \eta_2(\frac{(\eta_0)_N}{r}, f) + \eta_2(f, \frac{(\eta_0)_M}{r}).$$
Now we know that 
$$0 = (\eta \bullet \eta)_1(f) = \eta_1(\eta_1(f)) + \eta_2((\eta_0)_N, f) + \eta_2(f, (\eta_0)_M).$$
Since clearly $r\gamma = (\eta \bullet \eta)_1(f) = 0$, by regularity of $r$ we conclude that $\gamma = 0$ as desired.
\end{proof}

\section{Deformations of twisted objects}\label{parpardef}

Let $k$ be a commutative ground ring. We apply the results of \S \ref{parremcurv} to formal deformations ($R = k[[t]]$) in \S \ref{parformdef}. The more subtle case of $n$-th order infinitesimal deformations ($R = k[t]/t^{n+1}$) is treated in \S \ref{parinfdef}. More precisely, in both cases we describe constructions of uncurved deformations of twisted objects relative to a $cA_{\infty}$-deformation of an $A_{\infty}$-category. In the infinitesimal case, in order for our approach to work, the deformation of the original category is of one degree higher than the uncurved twisted object we construct. This throws some light on how the formal construction can be viewed as a $t$-adic limit of the infinitesimal constructions. In \S \ref{parinfty}, we investigate the derived interpretation of our constructions, in the special case where the original deformation is itself $A_{\infty}$ rather than $cA_{\infty}$. Finally, we briefly discuss the case of linear deformations (of which classical algebra deformations are a special case) in \S \ref{parlin}.

\subsection{Formal deformations}\label{parformdef}

Let $\AAA_0$ be a small $k$-quiver. Put $\AAA = k[[t]] \hat{\otimes}_k \AAA_0 = \AAA_0[[t]]$. The complex $\CC(\AAA_0)[[t]]$ inherits a natural brace algebra structure from $\CC(\AAA_0)$, for which the canonical map
$$\CC(\AAA_0)[[t]] \lra \CC(\AAA_0[[t]])$$
becomes an isomorphism of brace algebras. Using this isomorphism, we write a $cA_{\infty}$-structure $\mu$ on $\AAA = \AAA_0[[t]]$ as
$$\mu = \mu^{(0)} + \mu^{(1)}t + \mu^{(2)}t^2 + \dots + \mu^{(n)}t^n + \dots$$
for $\mu^{(k)} \in \CC^2(\AAA_0)$. Hence, the condition $\mu \bullet \mu = 0$ translates into relations between the $\mu^{(k)}$, of which the first relation is $\mu^{(0)} \bullet \mu^{(0)} = 0$.

For $M, N \in \Free(\AAA)$, we have
$$\widehat{\Hom}(M,N) = \Hom_0(M,N)[[t]]$$
and we write a connection on $M \in \Free(\AAA)$ as
$$\delta = \delta^{(0)} + \delta^{(1)}t + \delta^{(2)}t^2 + \dots + \delta^{(n)}t^n + \dots$$
for $\delta^{(k)} \in \Hom_0^1(M,M)$.
Hence, uncurvedness of $\delta$ translates into relations between the $\mu^{(k)}$ and $\delta^{(l)}$, of which the first relation is $\embr_{\delta^{(0)}}(\mu^{(0)}) = 0$.

Now consider $M \in \Free(\AAA_0)$ and connections $\delta, \psi \in \Hom_0^1(M,M)[[t]]$. Consider $\delta^{(0)}, \psi^{(0)} \in \Hom_0^1(M, M) \subseteq \Hom_0^1(M,M)[[t]]$.

\begin{lemma}\label{lemlem}
\begin{enumerate}
\item If for all $n \in \N$, $\delta^{(0)}$ (resp. $(\delta^{(0)}, \psi^{(0)})$) is locally $\mu^{(n)}$-allowable for the discrete topology, then $\delta$ (resp. $(\delta, \psi)$) is locally $\mu$-allowable for the $t$-adic topology.
\item Put $\Phi = \{ \mu^{(n)} \}_{n \in \N}$. If $\delta^{(0)}$ (resp. $(\delta^{(0)}, \psi^{(0)})$) is locally $\Phi$-allowable for the discrete topology, then $\delta^{(0)}$ (resp. $(\delta^{(0)}, \psi^{(0)})$) is locally $\mu$-allowable for the discrete topology.
\end{enumerate}
\end{lemma}

\begin{proof}
We prove the statement concerning $\delta$ in (1). Consider expression \eqref{eqnilp} for $\phi = \mu  = \sum_{k \in \N} \mu^{(k)} t^k$. This expression is equal to
\begin{equation}\label{eqnilpform}
\sum_{k \in \N} \sum_{k_{m-1}, \dots, k_1}\mu^{(k)}(c_l \otimes \dots \otimes c_{1} \otimes \delta_{jk_{m-1}} \otimes \dots \delta_{k_{d+1} k_d} \dots \otimes \delta_{k_1 i} \otimes b_{n} \otimes \dots \otimes b_1) t^k
\end{equation}
Let $n_0$ be given as in Definition \ref{defnilp}. Consider the elements $\delta^{(0)}_{k_{d+1}, k_d} \in \AAA_0(A_{k_d}, A_{k_{d+1}})$ as wel as $b^{(0)}_i \in \AAA_0(B_{i-1}, B_i)$ and $c^{(0)}_i \in \AAA_0(C_{i-1}, C_i)$. For every $k \in \N$, by assumption there is an $m_{0,k} \in \N$ such that
\begin{equation}\label{eqnilpform2}
\sum_{k_{m-1}, \dots, k_1}\mu^{(k)}(c^{(0)}_l \otimes \dots \otimes c^{(0)}_{1} \otimes \delta^{(0)}_{jk_{m-1}} \otimes \dots \delta^{(0)}_{k_{d+1} k_d} \dots \otimes \delta^{(0)}_{k_1 i} \otimes b^{(0)}_{n} \otimes \dots \otimes b^{(0)}_1)
\end{equation}
is equal to zero for every $m \geq m_{0,k}$.
Expression \eqref{eqnilpform} is equal to $\sum_{k \in \N} y_k$ with
$$y_k =  \sum_{\kappa_{1}, \dots, \kappa_{{n+m+l}}}\sum_{k_{m-1}, \dots k_1} x^{k_{m-1}, \dots k_1}_{\kappa_{1}, \dots, \kappa_{{n+m+l}}} t^{k + \sum_{i = 1}^{n +m +l} \kappa_{i}}$$
and $x^{k_{m-1}, \dots k_1}_{\kappa_{1}, \dots, \kappa_{{n+m+l}}}$ equal to
$$
\mu^{(k)}(c^{(\kappa_{n+m+l})}_l \otimes \dots \otimes c^{(\kappa_{n+m+1})}_{1} \otimes \delta^{(\kappa_{n+m})}_{jk_{m-1}} \otimes \dots \otimes \delta^{(\kappa_{n+1})}_{k_1 i} \otimes b^{(\kappa_n)}_{n} \otimes \dots \otimes b^{(\kappa_1)}_1).$$
Put $m_0 = n_0\cdot \max\{m_{0, k} \,|\, 0 \leq k \leq n_0\}$. For $m \geq m_0$, in $x^{k_{m-1}, \dots k_1}_{\kappa_{1}, \dots, \kappa_{{n+m+l}}}$ we have either at least $m_{0,k}$ consecutive entries coming from $\delta^{(0)}$, or else we have at least $n_0$ entries in total coming from $\delta^{(\kappa)}$ with $\kappa \geq 1$. Hence, for $k \leq n_0$, $y_k$ is a sum of expressions of type \eqref{eqnilpform2} with are zero by the definition of $m_{0, k}$, and expressions which are in $(t^{n_0})$. For $k \geq n_0$, we also have $y_k \in (t^{n_0})$. This finishes the proof.
\end{proof}

Put $\Phi = \{ \mu^{(n)} \}_{n \in \N}$. Consider $(M, \delta^{(0)}) \in \Tw(\AAA_0)$ and suppose $\delta^{(0)}$ is $\Phi$-allowable. We consider $\delta^{(0)} \in \Hom^1(M,M)$ as a connection on $M \in \Free(\AAA)$ and by Lemma \ref{lemlem}, we obtain $(M, \delta^{(0)}) \in \Tw(\AAA)$. We calculate the curvature
$$
c_M = \embr_{\delta^{(0)}}(\mu) = \sum_{k = 0}^{\infty}(\sum_{n = 0}^{\infty}\mu_k^{(n)}({\delta^{(0)}}^{\otimes k}) t^n) = \sum_{n = 0}^{\infty}c_M^{(n)} t^n
$$
for $$c_M^{(n)} = \sum_{k = 0}^{\infty} \mu_k^{(n)}({\delta^{(0)}}^{\otimes k}).$$
Now suppose $c_M^{(0)} = 0$, i.e. $(M, \delta^{(0)})$ is an uncurved object in $\Tw(\AAA_0)$. Then we have 
$$c_M = t\frac{c_M}{t} \hspace{0,5cm} \text{for} \hspace{0,5cm} \frac{c_M}{t} = \sum_{n = 0}^{\infty} c_M^{(n+1)} t^n.$$
Put $M_t = M \oplus \Sigma^{-1}M \in \Free(\AAA)$ endowed with the connection
$$\delta_{{M}_t} = \begin{pmatrix} \delta^{(0)} & \frac{c_M}{t} \\ -t & - \delta^{(0)} \end{pmatrix} \in \Hom^1({M}_t, {M}_t)$$
and consider the associated connection
$$\delta^{(0)}_{{M_t}} = \begin{pmatrix} \delta^{(0)} & c_M^{(1)} \\ 0 & - \delta^{(0)}  \end{pmatrix} \in \Hom^1_0({M}_t, {M}_t).$$
By definition, $\delta_{{M}_t}$ is an $\AAA$-defomation of $\delta^{(0)}_{{M_t}}$.
Since $t$ is regular with respect to $\Hom_0(M,M)[[t]]$, Proposition \ref{key} yields:

\begin{proposition}\label{propformdef}
Consider an uncurved object $(M, \delta^{(0)}) \in \Tw(\AAA_0)$ such that $\delta^{(0)}$ is $\Phi$-allowable.
We obtain uncurved objects $$M_t = (M \oplus \Sigma^{-1}M, \delta_{{M}_t}) = \cone(\frac{c_M}{t}, -t) \in \Tw(\AAA)$$ and $$M' = (M \oplus \Sigma^{-1}M, \delta^{(0)}_{{M}_t}) = \cone(c_M^{(1)}) \in \Tw(\AAA_0).$$ 
The result applies to the objects in the subcategory $\Tw_{ut}(\AAA_0)_{\infty} \subseteq \Tw(\AAA_0)$ of uncurved upper triangular objects in the sense of Example \ref{exiln}.
In particular, If $\AAA_0$ is an $A_{\infty}$-category, the result applies to the objects $A \in \AAA_0$ (endowed with the zero connection).
\end{proposition}

\subsection{Infinitesimal deformations}\label{parinfdef}
Let $\AAA_0$ be a small $k$-quiver. Put $\AAA =  k[t]/t^{n+1} \hat{\otimes}_k \AAA_0 = \AAA_0[t]/t^{n+1}$. In analogy with the formal case from \S \ref{parformdef}, we obtain a 
brace algebra isomorphism
$$\CC(\AAA_0)[t]/{t^{n+1}} \cong \CC(\AAA_0[t]/t^{n+1})$$
and we write a $cA_{\infty}$-structure $\mu$ on $\AAA_n = \AAA_0[t]/t^{n+1}$ as
$$\mu = \mu^{(0)} + \mu^{(1)}t + \mu^{(2)}t^2 + \dots + \mu^{(n)}t^n$$
for $\mu^{(k)} \in \CC^2(\AAA_0)$.

For $M, N \in \Free(\AAA_n)$, we have
$$\Hom_n(M,N) = \Hom_0(M,N)[t]/t^{n+1}$$
and we write a connection on $M \in \Free(\AAA_n)$ as
$$\delta = \delta^{(0)} + \delta^{(1)}t + \delta^{(2)}t^2 + \dots + \delta^{(n)}t^n$$
for $\delta^{(k)} \in \Hom_0^1(M,M)$.

For $M \in \Free(\AAA_0)$ and connections $\delta, \psi \in \Hom_n^1(M,M)$, we consider $\delta^{(0)}, \psi^{(0)} \in \Hom_0^1(M, M) \subseteq \Hom_n^1(M,M)$. In analogy with Lemma \ref{lemlem}, we have:
\begin{lemma}\label{lemlem2}
If for all $0 \leq k \leq n$, $\delta^{(0)}$ (resp. $(\delta^{(0)}, \psi^{(0)})$) is locally $\mu^{(k)}$-allowable, then $\delta^{(0)}$ (resp. $(\delta^{(0)}, \psi^{(0)})$) is locally $\mu$-allowable.
\end{lemma}

Consider $(M, \delta^{(0)}) \in \Tw(\AAA_0)$ and suppose $\delta^{(0)}$ is locally $\mu^{(l)}$-allowable for all $0 \leq l \leq n$. We consider $\delta^{(0)} \in \Hom_n^1(M,M)$ as a connection on $M \in \Free(\AAA_n)$ and by Lemma \ref{lemlem}, we obtain $(M, \delta^{(0)}) \in \Tw(\AAA_n)$. We calculate the curvature
$$
c_M = \embr_{\delta^{(0)}}(\mu) = \sum_{k = 0}^{\infty}(\sum_{l = 0}^n\mu_k^{(l)}({\delta^{(0)}}^{\otimes k}) t^l) = \sum_{l = 0}^nc_M^{(l)} t^l
$$
for $$c_M^{(l)} = \sum_{k = 0}^{\infty} \mu_k^{(l)}({\delta^{(0)}}^{\otimes k}).$$
Suppose $c_M^{(0)} = 0$, i.e. $(M, \delta^{(0)})$ is an uncurved object in $\Tw(\AAA_0)$. Then we have 
$$c_M = t\frac{c_M}{t} \hspace{0,5cm} \text{for} \hspace{0,5cm} \frac{c_M}{t} = \sum_{l = 0}^{n-1} c_M^{(l+1)} t^l.$$
Put $M_t = M \oplus \Sigma^{-1}M \in \Free(\AAA_n)$ endowed with the connection
$$\delta_{{M}_t} = \begin{pmatrix} \delta^{(0)} & \frac{c_M}{t} \\ -t & - \delta^{(0)} \end{pmatrix} \in \Hom_n^1({M}_t, {M}_t)$$
and consider the associated connection
$$\delta^{(0)}_{{M_t}} = \begin{pmatrix} \delta^{(0)} & c_M^{(1)} \\ 0 & - \delta^{(0)}  \end{pmatrix} \in \Hom^1_0({M}_t, {M}_t).$$
By definition, $\delta_{{M}_t}$ is an $\AAA$-defomation of $\delta^{(0)}_{{M_t}}$.
By Proposition \ref{key} we obtain
$$(M_t, \delta_{{M}_t}) = \cone(\frac{c_M}{t}, -t) \in \Tw(\AAA_n)$$ deforming the uncurved object
$$(M_t, \delta^{(0)}_{{M}_t}) = \cone(c_M^{(1)}) \in \Tw(\AAA_0),$$ 
but since $t$ is not regular in $k[t]/t^{n+1}$, we cannot conclude that $(M_t, \delta_{{M}_t})$ is uncurved.

\begin{example}\label{example}
The following example from \cite[Example 3.19]{kellerlowen} shows that the infinitesimal analogue of Proposition \ref{propformdef} is false. Let $(A_0 = k[u,v] = k\langle u, v \rangle/(uv -vu, v^2), \mu^{(0)}_2)$ be the free super commutative graded algebra on generators $u$ in degree $2$ and $v$ in degree $3$. Let $A = (A_0[\epsilon], \mu)$ be the first order deformation with non-zero contributions $\mu^{(1)}_0 = u$, $\mu^{(1)}_1(u) = v$, $\mu^{(1)}_1(v) = 0$. Then $c_A = u\epsilon$ and we can put $\frac{c_A}{\epsilon} = u$. Now $\mu_1(c_A) = \mu_1(u\epsilon) = v\epsilon^2 = 0$ but $\mu_1(\frac{c_A}{\epsilon}) = \mu_1(u) = v\epsilon \neq 0$.
\end{example}

Now suppose $\mu$ extends to an $n+1$-th order deformation of $\mu^{(0)}$ on $\AAA_{n+1} = \AAA_0[t]/t^{n+2}$:
$$\bar{\mu} = \mu^{(0)} + \mu^{(1)}t + \mu^{(2)}t^2 + \dots + \mu^{(n+1)}t^{n+1}.$$
Consider $\bar{M} = (M, \delta^{(0)}) \in \Tw(\AAA_{n+1})$. We have $c_{\bar{M}} = \sum_{l = 0}^{n+1} c_M^{(l)} t^l$
for $$c_M^{(l)} = \sum_{k = 0}^{\infty} \mu_k^{(l)}({\delta^{(0)}}^{\otimes k}).$$
We have 
$$c_{\bar{M}} = t\frac{c_{\bar{M}}}{t} \hspace{0,5cm} \text{for} \hspace{0,5cm} \frac{c_{\bar{M}}}{t} = \sum_{l = 0}^{n} c_M^{(l+1)} t^l.$$
Clearly, we can naturally interpret 
$$\frac{c_{\bar{M}}}{t} =  \sum_{l = 0}^{n} c_M^{(l+1)} t^l \in \Hom^2_n(M,M).$$

\begin{proposition}\label{propinfdef}
Suppose $\mu$ extends to an $n+1$-th order deformation $\bar{\mu}$ of $\mu^{(0)}$.
Consider an uncurved object $(M, \delta^{(0)}) \in \Tw(\AAA_0)$ such that $\delta^{(0)}$ is $\mu^{(l)}$-allowable for $0 \leq l \leq n+1$.
We obtain uncurved objects $$M_{(n)} = (M \oplus \Sigma^{-1}M, \delta_{{M}_{(n)}}) = \cone(\frac{c_{\bar{M}}}{t}, -t) \in \Tw(\AAA_n)$$ and $$M' = (M \oplus \Sigma^{-1}M, \delta_{M'} = \delta^{(0)}_{{M}_{(n)}}) = \cone(c_M^{(1)}) \in \Tw(\AAA_0).$$ 
The result applies to the objects in the subcategory $\Tw_{ut}(\AAA_0)_{\infty} \subseteq \Tw(\AAA_0)$ of uncurved upper triangular objects in the sense of Example \ref{exiln}.
In particular, If $\AAA_0$ is an $A_{\infty}$-category, the result applies to the objects $A \in \AAA_0$ (endowed with the zero connection).
\end{proposition}

\begin{proof}
We are to show that the curvature of $\delta_{{M}_{(n)}}$ vanishes. Put $\bar{\eta} = \embr_{\delta}(\bar{\mu}) = \sum_{k = 0}^{n+1} \eta^{(k)} t^k$ and $\eta = \embr_{\delta}(\mu) = \sum_{k = 0}^{n} \eta^{(k)} t^k$. 
We consider the object $$M_t = \cone(\frac{c_{\bar{M}}}{t}, -t) \in \Tw(\AAA_{n+1}).$$ Clearly, the object $M_t$ reduces to $M_{(n)} \in \Tw(\AAA_n)$, so it suffices to show that the curvature of $M_t$ is zero modulo $t^{n+1}$. 
On the one hand, according to Proposition \ref{key}, the curvature $c_{M_t}$ has the only possibly non-zero component given by:
$$\begin{aligned}
\bar{\eta}_1(\frac{c_{{\bar{M}}}}{t}) & = (\sum_{k = 0}^{n+1} \eta_1^{(k)} t^k)(\sum_{l = 0}^n c_M^{(l+1)} t^{l})\\
& = \sum_{p = 0}^{n+1} (\sum_{k + l = p} \eta_1^{(k)}(c_M^{(l+1)})) t^p.
\end{aligned}$$
On the other hand we know that
$$\begin{aligned}
0 = ((\bar{\eta} \bullet \bar{\eta})_0)_{\bar{M}} = \bar{\eta}_1(c_{{\bar{M}}}) & = (\sum_{k = 0}^{n+1} \eta_1^{(k)} t^k)(\sum_{l = 0}^{n} c_M^{(l+1)} t^{l+1}) \\
& = \sum_{q = 0}^{n+1} (\sum_{k + l +1 = q} \eta^{(k)}_1(c_M^{(l+1)})) t^q.
\end{aligned}$$
From comparing the coefficients in these two expressions we see that 
$$\bar{\eta}_1(\frac{c_{{\bar{M}}}}{t}) = \sum_{k + l = n+1} \eta_1^{(k)}(c_M^{(l+1)})) t^{n+1}.$$ 
which finishes the proof.
\end{proof}

\subsection{$A_{\infty}$-deformations}\label{parinfty}
Put $R = k[[t]]$. Let $(\AAA_0, \mu^{(0)})$ be a small $k$-linear $A_{\infty}$-category and let $(\AAA = \AAA_0[[t]], \mu)$ be an $R$-linear $A_{\infty}$-category deforming $\AAA_0$. 
Consider an object $(M, \delta^{(0)}) \in \Tw_{ut}(\AAA_0)_{\infty}$.
Let $$(M_t, \delta_{{M}_t}) = \cone(\frac{c_M}{t}, -t) \in \Tw(\AAA)$$ 
be the uncurved object obtained form Proposition \ref{propformdef}.
An uncurved twisted object $N$ over an $A_{\infty}$-category $\BBB$ naturally determines an $A_{\infty}$-module $\tilde{N}$ over $\BBB$ ($\BBB$-module for short) with $\tilde{N}(B) = \Hom(B, N)$ for $B \in \BBB$. See for instance \cite{lefevre}.
We thus consider the $\AAA_0$-module $\tilde{M}$ and the $\AAA$-module $\tilde{M_t}$. Further, the reduction map $\AAA \lra \AAA_0$ allows us to view $\tilde{M}$ as an $\AAA$-module $\tilde{M}_{\AAA}$. The natural maps
$$\Hom(A, M_t) \cong \Hom(A, M) \oplus \Hom(A, \Sigma^{-1}M) \lra \Hom_0(A, \Sigma^{-1}M)$$
given by projection onto the second factor followed by reduction determine a morphism of $\AAA$-modules
$$\theta_M: \tilde{M}_t \lra \Sigma^{-1}\tilde{M}_{\AAA}.$$

\begin{proposition}\label{propder}
\begin{enumerate}
\item For every $M \in \Tw_{ut}(\AAA_0)_{\infty}$, the morphism $\theta_M$ is a quasi-isomorphism of $\AAA$-modules. 
\item If $\tilde{M} \in D(\AAA_0)$ is perfect, the same holds for $\tilde{M}_t \cong \Sigma^{-1}\tilde{M}_{\AAA} \in D(\AAA)$.
\item The object $\tilde{M}_t \cong \Sigma^{-1}\tilde{M}_{\AAA} \in D(\AAA)$ is a derived lift of $\tilde{\cone}(c^{(1)}_M) \in D(\AAA_0)$, i.e. $k \otimes^L_R \Sigma^{-1}\tilde{M}_{\AAA} = k \otimes_R \tilde{M}_t = \tilde{\cone}(c^{(1)}_M)$.
\end{enumerate}
\end{proposition}

\begin{proof}
(1) First, we look at the claim for $A \in \AAA$, endowed with the zero connection.
We thus have $A_t = A \oplus \Sigma^{-1}A$ endowed with
$$\delta_{{A}_t} = \begin{pmatrix} 0 & 0 \\ -t & 0 \end{pmatrix}.$$
On the other hand, for $-t \in \Hom^1(A, \Sigma^{-1}{A})$, we have $\cone(-t) = \Sigma^{-1}{A} \oplus {A}$ endowed with
$$\delta_{\cone(-t)} = \begin{pmatrix} 0 & -t \\ 0 & 0 \end{pmatrix}.$$
Hence, there is a canonical isomorphism $A_t \cong \cone(-t)$ given by permuting the two factors of the direct sum.
It then easily follows that for the induced $\AAA$-modules, $\tilde{A}_t = \Hom(-, {A}_t)$ is isomorphic to the cone of 
$$-t: \Sigma^{-1}{\AAA(-,A)} \lra \Sigma^{-1}{\AAA(-,A)}.$$
On the other hand, the exact sequence
$$\xymatrix{ 0 \ar[r] & R \ar[r]_{-t} & R \ar[r] & k \ar[r] & 0}$$
gives rise to an exact sequence of $\AAA$-modules
\begin{equation} \label{eqexact}
\xymatrix{0 \ar[r] & \Sigma^{-1}{\AAA(-,A)} \ar[r]_-{-t} & \Sigma^{-1}{\AAA(-,A)} \ar[r] & \Sigma^{-1}{\AAA_0(-,A)} \ar[r] & 0.}
\end{equation}
Consequently, we obtain a quasi-isomorphism of $\AAA$-modules
$$\cone(-t) = \Sigma^{-1}\AAA(-,A) \oplus \AAA(-,A) \lra \Sigma^{-1}\AAA_0(-,A)$$
given by projection onto the first factor followed by reduction.
It follows that the isomorphic morphism $\theta_A$ is a quasi-isomorphism too.

To finish the proof of (1), it now suffices to show that the objects $M$ in $\Tw_{ut}(\AAA_0)$ for which $\theta_M$ is a quasi-isomorphism are closed under taking direct sums and cones. 
First, consider such objects $M, N \in \Tw_{ut}(\AAA_0)$ and a closed morphism $f \in \Hom_0^0(M,N)$, and consider $f$ as a (not necessarily closed) morphism $\ovl{f} \in \Hom^0({M}, {N})$.
We can extend the diagram from Lemma \ref{lemconee} applied to $\ovl{f}$ to a diagram of $\BBB$-modules
$$\xymatrix{{\tilde{M}_t = {M} \tilde{\oplus} \Sigma^{-1} {M}} \ar[d]_{\ovl{f}_t} \ar[r] & {\Sigma^{-1} \tilde{M}} \ar[d]^{-\ovl{f}} \ar[r] & {\Sigma^{-1}\tilde{M}_{\AAA}} \ar[d]^{-f} \\
{\tilde{N}_t = {N} \tilde{\oplus} \Sigma^{-1}{N}} \ar[r] \ar[d] & {\Sigma^{-1} \tilde{N}} \ar[d] \ar[r] & {\Sigma^{-1}\tilde{N}_{\AAA}} \ar[d]\\ 
{\cone(\ovl{f}_t) \cong_{\sigma} \cone(\ovl{f})_t} \ar[r] & {\Sigma^{-1} \cone(\ovl{f})} \ar[r] & {\Sigma^{-1} \cone(f)_{\AAA}}, }$$
where in the middle column we encounter $cA_{\infty}$-modules rather than $A_{\infty}$-modules. 
Forgetting about the middle column, the horizontal maps from top to bottom are precisely $\theta_M$, $\theta_N$ and $\theta_{\cone(f)}$.
It follows that if $\theta_M$ and $\theta_N$ are quasi-isomorphisms, the same holds for $\theta_{\cone(f)}$.
The argument for direct sums is similar.

(2) From the exact sequences \eqref{eqexact}, we deduce that the modules $\AAA_0(-,A)_{\AAA}$ are perfect $\AAA$-modules. It readily follows that every perfect $\AAA_0$-module remains perfect when considered as an $\AAA$-module. 

(3) It suffices to show that for the twisted object ${M}_t$, the corresponding $\AAA$-module $\tilde{M}_t$ is cofibrant over $R$. By construction, the values $\Hom(A, {M}_t)$ are graded free over $R$. Since $R$ is regular, this ensures cofibrancy of the complexes.
\end{proof}

\begin{remark}\label{remint}
In the setup of infinitesimal deformations considered in \S \ref{parinfdef}, analogous results to Proposition \ref{propder} do not hold. In particular, for $\AAA = \AAA_0[t]/t^{n+1}$, the construction of $M_t \in \Tw(\AAA)$ does not guarantee that the $\AAA$-module $\tilde{M}_t$ is $k[t]/t^{n+1}$-cofibrant, so it is not necessarily a derived lift. The conclusion that $\tilde{M}_t$ is a derived lift of  $\tilde{\cone}(c^{(1)}_M)$ can however be taken if all relevant $k[t]/t^{n+1}$-modules are (cohomologically) bounded above.
\end{remark}

\subsection{Linear deformations}\label{parlin}
Let $\AAA_0$ be a $k$-linear category, i.e. $\AAA_0(A,A') = \AAA_0^0(A,A')$ and  $\mu^{(0)} = \mu_2^{(0)}$. In this case every $R$-linear $cA_{\infty}$-deformation is necessarily an $R$-linear category structure $\mu$ on $\AAA = R \hat{\otimes}_k \AAA_0$, i.e $\mu = \mu_2$, which is a forteriori uncurved.
A twisted object $(M = \oplus_i \Sigma^{n_i} A_i, \delta_M) \in \Tw(\AAA)$ can be written as $M = \oplus_{n \in \Z} (\oplus_{i \in I_{n}} \Sigma^{n}A_i)$ with $\delta_M$ determined by the non-zero entries
$$\delta_M^n \in\Hom^1(\oplus_{i \in I_n} \Sigma^nA_i, \oplus_{i \in I_{n+1}} \Sigma^{n+1}A_i) \cong \Hom^0(\oplus_{i \in I_n} A_i, \oplus_{i \in I_{n+1}} A_i).$$ 
Note that by Example \ref{exmun}, every connection actually determines a twisted object.
The $cA_{\infty}$-structure $\eta = \mathrm{embr}_{\delta}(\mu)$ on $\Tw(\AAA)$ reduces to
$$c_M = (\eta_0)_M = \mu(\delta_M, \delta_M);$$
$$\eta_1 = \mu(\delta_N \otimes 1) + \mu(1 \otimes \delta_M);$$
$$\eta_2 = \mu_2.$$
Put $R = k[[t]]$. For $(M, \delta^{(0)}) \in \Tw(\AAA_0)$, the object
$$({M}_t, \delta_{M_t}) = \cone(\frac{c_M}{t}, -t)$$ from Proposition \ref{propformdef} has
\begin{equation}\label{eqdesdelta}
\delta_{M_t} = \begin{pmatrix}   \delta^{(0)} & \frac{\mu(\delta^{(0)}, \delta^{(0)})}{t} \\ -t & - \delta^{(0)} \end{pmatrix}.
\end{equation}
As in \S \ref{parinfty}, we can interpret the twisted objects as modules. Concretely, an object $N \in \Tw(\AAA)$ has an associated precomplex $\tilde{N}$ of free $\AAA$-modules, and if $N$ is uncurved, then $\tilde{N}$ is an actual complex. Thus, in our situation we obtain the complex $\tilde{M}_t$ of free $\AAA$-modules. The fact that $\tilde{M}_t$ is a complex amounts to the computation
$$\eta_2(\delta_t, \delta_t) = \mu(\delta_t, \delta_t) = 0$$
which is a special case of the proof of Proposition \ref{propformdef}.

 Put $R = k[t]/t^{n+1}$. Under the hypotheses of Proposition \ref{propinfdef}, we once again obtain the uncurved object $({M}_t, \delta_{M_t}) \in \Tw(\AAA)$ with the same description of $\delta_{M_t}$ as above in \eqref{eqdesdelta}. Since we no longer have the exact sequences \eqref{eqexact}, the derived interpretation of $M_t$ as given in Proposition \ref{propder}(1) no longer holds. We end this section with some observations which still hold for infinitesimal deformations in the specific setup of this section, see Remark \ref{remint}.

\begin{proposition}
\begin{enumerate}
\item If $\tilde{M}$ is bounded (resp. bounded above, resp. bounded below), the same holds for $\tilde{M}_t$.
\item If $\tilde{M}$ is a bounded complex of finite free $\AAA_0$-modules, then $\tilde{M}_t$ is a bounded complex of finite free $\AAA$-modules.
\item If $\tilde{M}$ is bounded above, then the object $\tilde{M}_t \in D(\AAA)$ is a derived lift of $\tilde{\cone}(c^{(1)}_M) \in D(\AAA_0)$, i.e. $k \otimes^L_R \tilde{M}_t = k \otimes_R \tilde{M}_t = \tilde{\cone}(c^{(1)}_M)$.
\end{enumerate}
\end{proposition}

\section{Torsion derived categories}\label{parpartor}

Put $R = k[[t]]$. In this section we introduce the torsion derived category of an $R$-cofibrant dg category $\AAA$ in \S \ref{partor}. For us, the torsion derived category is mainly a technical tool which allows us to prove our main result Theorem \ref{thmmain} in \S \ref{parcurv}. Putting $\AAA_0 = k \otimes_R \AAA$, the forgetful functor $\iota: D(\AAA_0) \lra D(\AAA)$ has a left adjoint $k \otimes^L_R -$ and a right adjoint $\RHom_R(k,-)$ for which $\Kern(k \otimes^L_R -) = \Kern(\RHom_R(k,-))$. In contrast with the infinitesimal situation ($R = k[t]/t^n$), in general this kernel is non-zero, so we cannot apply a suitable version of Nakayama's Lemma. Roughly speaking, for our purpose, we force a derived version of Nakayama's Lemma to hold by working with the torsion derived category $D_{tor}(\AAA) \subseteq D(\AAA)$, which is equivalent to the quotient $D(\AAA)/\Kern(\RHom_R(k,-))$. In \S \ref{parcompgen}, we prove some basic results, like the fact that $D_{tor}(\AAA)$ is compactly generated by the objects $\iota(A)$ for $A \in \AAA$ (Proposition \ref{propcompgen}), and in \S \ref{partormor} we introduce the notion of torsion Morita equivalence. 

The results in this section are related to various works on torsion and complete derived categories, which have been developed in different contexts and levels of generality, see for instance \cite{alonsojeremiaslipman}, \cite{kashiwarashapira}, \cite{portashaulyekutieli}, \cite{PSY2}, \cite{positselski3}, \cite{positselski:contra}, \cite{positselski:MGM}. Although we only develop a rudimentary torsion side of the story here, it is likely that a richer picture including an appropriate version of MGM duality exists in this setting. The details of this picture will be investigated elsewhere.

\subsection{The torsion derived category}\label{partor}

Let $\AAA$ be an $R$-cofibrant dg category.
The external Hom and tensor functors
$$- \otimes_R -: \Mod(R) \times \Mod(\AAA) \lra \Mod(\AAA): (X, M) \longmapsto X \otimes_R M$$
$$\Hom_R(-,-): \Mod(R) \times \Mod(\AAA) \lra \Mod(\AAA): (X, M) \longmapsto \Hom_R(X, M)$$
give rise to derived functors
$$- \otimes^L_R -: D(R) \times D(\AAA) \lra D(\AAA): (X, M) \longmapsto X \otimes^L_R M$$
$$\RHom_R(-,-): D(R) \times D(\AAA) \lra D(\AAA): (X, M) \longmapsto \RHom_R(X, M)$$
which are balanced, i.e. they can be calculated in either argument.

\begin{lemma}\label{lembas}
Let $X \in \per(R)$ be a perfect object and consider $M, N \in D(\AAA)$. We have
$$X \otimes^L_R \RHom_{\AAA}(M,N) = \RHom_{\AAA}(M, X \otimes^L_R N) = \RHom_{\AAA}(\RHom_R(X,M), N).$$
\end{lemma}

\begin{proof}
All three expressions define exact functors in $X$ which agree on $X = R$.
\end{proof}

Consider $M \in \Mod(\AAA)$, $A \in \AAA$ and $X = R/t^n$. We have
$$\Hom_R(R/t^n, M)^i(A)  \cong \{ x \in M^i(A) \,\, |\,\, t^nx = 0\} \subseteq M^i(A).$$
The natural projection map $R/t^n \lra R/t^{n-1}$ induces a morphism
$$\Hom_R(R/t^{n-1}, M) \lra \Hom_R(R/t^{n}, M)$$
which corresponds to the natural inclusion between submodules of $M$.
We thus obtain the \emph{torsion submodule} of $M$, which is
$$\Gamma(M) = \colim_{n \geq 0} \Hom_R(R/t^n, M).$$
The torsion functor $\Gamma: \Mod(\AAA) \lra \Mod(\AAA)$ is left exact and we obtain its right derived functor
$R\Gamma: D(R) \lra D(R)$ which satisfies
$$R\Gamma(M) = \colim_{n \geq 0} \RHom_{R}(R/t^n, M).$$

\begin{definition}
Let $M \in D(\BBB)$.
\begin{enumerate}
\item $M$ is \emph{derived torsionfree} if $M \in \Kern(R\Gamma)$.
\item $M$ is \emph{derived torsion} if $\RHom_{\BBB}(M,N) = 0$ for all $N$ with $R\Gamma(N) = 0$. 
\end{enumerate}
\end{definition}

We thus obtain the full subcategory $D_{tor}(\AAA) = ^{\perp}\Kern(R\Gamma) \subseteq D(\AAA)$ of torsion modules, which we call the \emph{torsion derived category}. Clearly, $D_{tor}(\AAA)$ is a triangulated subcategory closed under coproducts.

For $M \in D(\AAA)$, there is a canonical morphism $\varphi_M: R\Gamma(M) \lra M$ induced by the morphisms $\RHom_R(R/t^{n-1}, M) \lra \RHom_R(R, M) = M$. 

\begin{lemma}\label{lemidpot}
For $M \in D(\AAA)$, the morphism $R\Gamma(\varphi_M): R\Gamma(R\Gamma(M)) \lra R\Gamma(M)$ is an isomorphism.
\end{lemma}

\begin{lemma}\label{lemnak}
The following are equivalent for $M \in D(\AAA)$:
\begin{enumerate}
\item $R\Gamma(M) = 0$;
\item $\RHom_R(R/t^n, M) = 0$ for all $n \geq 0$;
\item $\RHom_R(R/t^n, M) = 0$ for some $n \geq 0$;
\item $\RHom_R(k, M) = 0$.
\end{enumerate}
\end{lemma}

\begin{proof}
This follows from a derived version of Nakayama's Lemma.
\end{proof}

For ${M} = ({M}, d) \in D({\AAA})$, put ${M}_{t^n} = {M} \oplus \Sigma^{-1}{M}$ with differential
$$\begin{pmatrix} d & 0 \\ -t^n & -d \end{pmatrix}.$$
Clearly, ${M}_{t^n} \cong \Sigma^{-1} \cone(t^n: {M} \lra {M})$.

\begin{lemma}\label{lemcone}
We have $$\RHom_R(R/t^n, {M}) = {M}_{t^n} = \Sigma^{-1}R/t^n \otimes^L_R {M}.$$ The canonical map ${M} \lra R/t^n \otimes^L_R {M}$ is given by injection from the second factor ${M} \cong \Sigma \Sigma^{-1} {M} \lra \Sigma {M}_{t^n}$ and the canonical map $\RHom_R(R/t^n, {M}) \lra {M}$ is given by projection on the first factor ${M}_{t^n} \lra {M}$.
\end{lemma}

\begin{proof}
This follows by using the exact sequence
$$\xymatrix{ 0 \ar[r] & {R} \ar[r]_{t} & {R} \ar[r] & R/t^n \ar[r] & 0}.$$
\end{proof}

\begin{corollary}
The following are equivalent for $M \in D(\AAA)$:
\begin{enumerate}
\item $R\Gamma(M) = 0$;
\item $R/t^n \otimes^L_R M = 0$ for all $n \geq 0$;
\item $R/t^n \otimes^L_R M = 0$ for some $n \geq 0$;
\item $k \otimes^L_R M = 0$.
\end{enumerate}
\end{corollary}

\begin{proof}
This follows from Lemma \ref{lemnak} and Lemma \ref{lemcone}.
\end{proof}

\begin{lemma}\label{key2}
Let $M, N \in D(\AAA)$. We have
$$\begin{aligned}
\RHom_{\AAA}(R\Gamma(M), N) & = \RHom_{\AAA}(\colim_{n \geq 0} \RHom_R(R/t^n, M), N) \\
& = \lim_{n \geq 0} \RHom_{\AAA}(\RHom_R(R/t^n, M), N)\\
& = \lim_{n \geq 0} \RHom_{\AAA}(M, R/t^n \otimes^L_R N)\\
\end{aligned}$$
\end{lemma}

\begin{proof}
This follows from Lemma \ref{lembas}.
\end{proof}

\begin{lemma}\label{lemcc}
Let $M, N \in D(\AAA)$.
\begin{enumerate}
\item $R\Gamma(M) \in D_{tor}(\AAA)$.
\item $M \in D_{tor}(\AAA)$ if and only if $\varphi_M: R\Gamma(M) \lra M$ is an isomorphism.
\end{enumerate}
\end{lemma}

\begin{proof}
(1) Consider $N \in \Kern(R\Gamma)$. By Lemma \ref{key2} we have $\RHom_{\AAA}(R\Gamma(M), N) =  \lim_{n \geq 0} \RHom_{\BBB}(M, R/t^n \otimes^L_R N) = 0$.
(2) If $\varphi_M$ is an isomorphism then $M \in D_{tor}(\AAA)$ by (1). Conversely, if $M \in D_{tor}(\AAA)$ then also $\cone(\varphi_M) \in D_{tor}(\AAA)$. But by Lemma \ref{lemidpot}, $R\Gamma(\cone(\varphi_M)) = 0$. Hence $\cone(\varphi_M) = 0$ as desired.
(3) For $M \in D_{tor}(\AAA)$ we have $\RHom_{\AAA}(M, N) \cong \RHom_{\AAA}(R\Gamma(M), N) \cong L \Lambda(\RHom_{\AAA}(M, N))$ by Lemma \ref{key2}.
\end{proof}

Consider the Verdier quotient $D(\AAA)/\Kern(R\Gamma)$. We obtain an induced functor $R\Gamma: D(\AAA)/\Kern(R\Gamma) \lra D_{tor}(\AAA)$.

\begin{proposition}
The composition $D_{tor}(\AAA) \lra D(\AAA) \lra D(\AAA)/\Kern(R\sigma)$ is an equivalence of categories with inverse equivalence given by $R\Gamma: D(\AAA)/\Kern(R\Gamma) \lra D_{tor}(\AAA)$.
\end{proposition}

\subsection{Compact objects and generators}\label{parcompgen}

Let $\AAA$ be an $R$-cofibrant dg category and put $\AAA_0 = k \otimes_R \AAA$. The forgetful functor $\iota: D(\AAA_0) \lra D(\AAA)$
has a left adjoint $L\tau: D(\AAA) \lra D(\AAA_0)$ and a right adjoint $R\sigma: D(\AAA) \lra D(\AAA_0)$ with
$\iota L\tau = k \otimes^L_R -$ and $\iota R\sigma = \RHom_R(k, -)$. According to Lemma \ref{lemcone}, we have 
\begin{equation}\label{eqshift}
\iota R\sigma \cong \Sigma^{-1} \iota L\tau.
\end{equation}

\begin{proposition}\label{proptor}
\begin{enumerate}
\item For $X \in D(\AAA_0)$, $\iota X \in D_{tor}(\AAA)$.
\item If the objects $X_i \in \Mod(\AAA_0)$ generate $D(\AAA_0)$, then the objects $\iota X_i \in \Mod(\AAA)$ generate $D_{tor}(\AAA)$.
\end{enumerate}
\end{proposition}

\begin{proof}
(1) For $M \in D(\AAA)$ we have
$$\RHom_{\AAA}(\iota X, M) = \RHom_{\AAA_0}(X, R\sigma M).$$
Hence, $\iota X \in D_{tor}(\AAA)$. (2) If $M \in D_{tor}(\AAA)$ satisfies $\RHom_{\AAA}(\iota X_i, M) = 0$, it follows that $R\sigma M = 0$ since $(X_i)_i$ generates $D(\AAA_0)$. But then $\RHom_{\AAA}(M,M) = 0$ whence $M = 0$.
\end{proof}

\begin{lemma}\label{lemcoprod}
The functor $R\sigma: D(\AAA) \lra D(\AAA_0)$ preserves coproducts.
\end{lemma}

\begin{proof}
For a collection of objects $M_i \in D(\AAA)$, consider the comparison map $\lambda: \oplus_i R\sigma(M_i) \lra R\sigma (\oplus_i M_i)$ in $D(\AAA_0)$. Since $\iota: D(\AAA_0) \lra D(\AAA)$ reflects isomorphisms, it suffices that $ \iota(\lambda)$ is an isomorphism in $D(\AAA)$. But this is indeed the case by \eqref{eqshift}.
\end{proof}

\begin{proposition}\label{propcompgen}
\begin{enumerate}
\item The functor $\iota: D(\AAA_0) \lra D(\AAA)$ preserves compact objects.
\item If the objects $X_i$ compactly generate $D(\AAA_0)$, then the objects $\iota X_i$ compactly  generate $D_{tor}(\AAA)$ inside $D(\AAA)$.
\end{enumerate}
\end{proposition}

\begin{proof}
(1) follows from Lemma \ref{lemcoprod} and adjunction. (2) follows from (1) and Proposition \ref{proptor} (2).
\end{proof}

\begin{lemma}\label{lemhomtor}
If $N \in D({\AAA})$ is compact and $M \in D_{tor}({\AAA})$, then $\RHom_{{\AAA}}(N,M) \in D_{tor}(R)$.
\end{lemma}

\begin{proof}
First we take $M = \iota X$ for $X \in D(\AAA_0)$ and $N = A \in {\AAA}$. Then $\RHom_{{\AAA}}({A}, \iota X) = \iota (X(A)) \in D_{tor}(R)$. Since $A$ is compact, $\RHom_{\AAA}(A, -)$ is triangulated and preserves coproducts. Hence, by Proposition \ref{proptor} (2), it follows that $\RHom_{\AAA}(A, M) \in D_{tor}(R)$ for all $M \in D_{tor}(\AAA)$. The claim follows by a similar reasoning for the first argument.  \end{proof}

\begin{definition}\label{deftorsion}
An $R$-linear dg category ${\AAA}$ is called \emph{torsion} if $D_{tor}({\AAA}) = D({\AAA})$.
\end{definition}

Let ${\AAA}_t \subseteq D_{dg}({\AAA})$ be the full dg subcategory spanned by the objects ${A}_t$ for $A \in \AAA$. Put $A' = A \oplus \Sigma^{-1}A \in D(\AAA_0)$ and let $\AAA' \subseteq D_{dg}(\AAA_0)$ be the full dg subcategory spanned by the objects $A'$ for $A \in \AAA$. 

\begin{proposition}\label{propattor}
The $R$-cofibrant dg category ${\AAA}_t$ is torsion and satisfies $k \otimes_R \AAA_t = \AAA'$.
\end{proposition}

\begin{proof}
We have a commutative square
$$\xymatrix{ {D({\AAA}_t)} \ar[r] \ar[d]_{L\tau} & {D({\AAA})} \ar[d]^{L\tau} \\ {D(\AAA')} \ar[r] & {D(\AAA_0)} }$$
in which the horizontal maps are the tensor functors induced by the inclusions of subcategories ${\AAA}_t \subseteq D_{dg}({\AAA})$ and $\AAA' \subseteq D_{dg}(\AAA_0)$.
Now suppose $M \in {D({\AAA}_t)}$ satisfies $L\tau M = 0 \in D(\AAA')$. Then the image of $M$ in $D_{tor}({\AAA}) \subseteq D({\AAA})$ satisfies $L\tau M = 0 \in D(\AAA_0)$ and hence also $R\sigma M = 0$. It follows that $M = 0$. As a consequence, $D_{tor}({\AAA}_t) = D({\AAA}_t)$ as desired.
\end{proof}

\subsection{Torsion Morita equivalences}\label{partormor}
In this section we describe how a derived Morita equivalence can be lifted to a derived torsion Morita equivalence under deformation.
Let ${\AAA}$, ${\BBB}$ be cofibrant $R$-linear dg categories and $M$ a cofibrant $\AAA$-$\BBB$-bimodule (taking values $M(B, A) \in \Mod(R)$). Put $\AAA_0 = k \otimes_R {\AAA}$, $\BBB_0 = k \otimes_R {\BBB}$, $M_0 = k \otimes_R {M}$.

We consider the tensor functors 
$${M}_0 \otimes^L_k -: D({\AAA}_0) \lra D({\BBB}_0): {A} \longmapsto {M_0}(-, {A})$$
and
$${M} \otimes^L_R -: D({\AAA}) \lra D({\BBB}): {A} \longmapsto {M}(-, {A}).$$

\begin{definition}\label{deftormor}
\begin{enumerate}
\item $M_0$ (resp. ${M}$) is a \emph{Morita bimodule} if $M_0 \otimes^L_k -$ (resp. ${M} \otimes^L_R -$) is an equivalence of categories.
\item ${M}$ is a \emph{torsion Morita bimodule} if ${M} \otimes^L_k -$ induces an equivalence of categories
$${M} \otimes^L_R -: D_{tor}({\AAA}) \lra D_{tor}({\BBB}).$$
\end{enumerate}
\end{definition}

The aim of this section is to prove the following:

\begin{proposition}\label{proptorlift}
If $M_0$ is a Morita bimodule, then ${M}$ is a torsion Morita bimodule.
\end{proposition}
\begin{proof}
By Lemma \ref{lemcone} and Proposition \ref{propcompgen}(2), the category $D_{tor}({\AAA})$ is compactly generated by the objects $${A}_t = \Sigma^{-1}\cone(t: {A} \lra {A})$$
in $D({\AAA})$.
The object ${M}(-, {A})$ is cofibrant in $D({\BBB})$ and satisfies $k \otimes_R {M}(-, {A}) = M_0(-, A)$.
On the other hand $${M}(-,{A}_t) = \Sigma^{-1}\cone(t: {M}(-, {A}) \lra {M}(-, {A})) = {M}(-, {A})_t.$$
Since $M_0$ is a Morita bimodule, the objects $M_0(-,A)$ compactly generate $D(\BBB_0)$. Hence, by Proposition \ref{propcompgen} (2), the objects ${M}(-, {A})_t$ compactly generate $D_{tor}({\BBB})$ inside $D({\BBB})$.
It remains to show that the morphisms
$$\lambda: {\AAA}({A}_t, C_t) \lra \Hom_{{\BBB}}({M}(-, {A}_t), {M}(-, C_t))$$
are quasi-isomorphisms for $A, C \in \AAA$. 
Put $A' = A \oplus \Sigma^{-1}A$, $C' = C \oplus \Sigma^{-1}C \in D(\AAA_0)$.
We have $k \otimes_R {A}_t = A'$, $k \otimes_R {C}_t = C'$, $k \otimes_R {M}(-, {A}_t) = M_0(-,A')$, and using lemma \ref{lembas} and adjunction, $k \otimes_R^L \lambda$ is given by the canonical morphism
$$\lambda': \AAA_0(A', C') \lra \Hom_{\BBB_0}(M_0(-, A'), M_0(-, C')).$$
By assumption, $\lambda'$ is a quasi-isomorphism. By Lemma \ref{lemhomtor}, the domain and codomain of $\lambda$ are in $D_{tor}(R)$ and hence the same holds for $\cone(\lambda)$. On the other hand, $k \otimes^L_R \cone(\lambda) = \cone(k \otimes_R^L \lambda) = \cone(\lambda') = 0$. Since $\Kern(R\sigma) = \Kern(\RHom_R(k,-)) =\Kern(k \otimes_R^L -)$, it follows that $\cone(\lambda) = 0$ and $\lambda$ is a quasi-isomorphism as desired.
\end{proof}

\section{The curvature problem}\label{parcurv}

In this section, we solve the curvature problem for formal deformations of small dg categories, that is, the phenomenon that Maurer-Cartan elements in the formal Hochschild complex naturally parameterize \emph{curved} $A_{\infty}$-deformations rather than dg or $A_{\infty}$-deformations. Precisely, we show that for a cohomologically bounded above dg category, these Maurer-Cartan elements up to gauge equivalence are in bijection with Morita deformations up to torsion Morita equivalence. This result is stated in Theorem \ref{thmmain}, and an infinitesimal counterpart is stated in Theorem \ref{thmmaininf}. The formal result is illustrated by the example of the graded field in \S \ref{pargraded}.

\subsection{Transport of $cA_{\infty}$-structures}\label{partransp}Put $R = k[[t]]$.
Let $\AAA_0 = (\AAA_0, \mu^{(0)})$ be a small $k$-linear $A_{\infty}$-category and put ${\AAA} = R \hat{\otimes}_k \AAA_0 = \AAA_0[[t]]$. 
Consider
$$\mathrm{MC}(A_0; k[[t]]) = \{ \phi \in \Sigma\CC(A_0)[[t]]^1 \,\,| \,\, [\mu^{(0)}, \phi t] + \phi t \bullet \phi t = 0 \}/ \sim,$$
the set of elements $\phi$ of shifted degree $1$ for which $\phi t \in \Sigma\CC(A_0)[[t]]$ is a solution of the Maurer Cartan equation up to gauge equivalence of the elements $\phi t$. 
For details on the pronilpotent MC formalism, the reader may wish to consult for instance \cite{yekutieli:MC}.

Let $\tw(\AAA_0)$ be the classical category of (uncurved, finite) twisted complexes over $\AAA_0$. For each $(M, \delta^{(0)}_M) \in \tw(\AAA_0)$, let
$${\delta}_{{M}} = \delta^{(0)}_M + \delta^{(1)}_Mt + \dots + \delta^{(n)}_Mt^n + \dots$$
be a fixed connection which defines an object $({M}, {\delta}_{{M}}) \in \Tw(\AAA)$. We choose $\delta_{{A}} = 0$ for $A \in \AAA_0$. Let $\ovl{\tw(\AAA_0)} \subseteq \Tw(\AAA)$ be the full subquiver consisting of the objects $({M}, \delta_M)$ for $(M, \delta^{(0)}_M) \in \tw(\AAA_0)$. We thus have $\AAA \subseteq \ovl{\tw(\AAA_0)}$.

We have a diagram:
$$\xymatrix{ {\Sigma \CC (\tw(\AAA_0))[[t]]} \ar[d]_{\pi'} \ar[r]^t & {\Sigma \CC (\ovl{\tw(\AAA_0)})} \ar[d]_{\pi} \ar[r] & {\Sigma \CC(\tw(\AAA_0))} \ar[d]_{\pi_0} \\
{\Sigma\CC(\AAA_0)[[t]]} \ar[r]_t & {\Sigma \CC(\AAA)} \ar[r] & {\Sigma \CC(\AAA_0)} }$$

The restriction maps $\pi_0$ and $\pi'$ are induced by the inclusion $\AAA_0 \subseteq \tw(\AAA_0)$ and $\pi$ is induced by $\AAA \subseteq \ovl{\tw(\AAA_0)}$.

The maps $\mathrm{embr}_{\delta^{(0)}}$ and $\mathrm{embr}_{\delta}$ define $B_{\infty}$-sections of $\pi_0$ and $\pi$ respectively (see \cite{lowen9}), making the resulting right hand square commute. We thus obtain an induced $B_{\infty}$-section $s_{\delta}: \Sigma\CC(\AAA_0)[[t]] \lra \Sigma \CC (\tw(\AAA_0)[[t]]$ of $\pi'$. Now $\pi_0$ and $\mathrm{embr}_{\delta^{(0)}}$ are in fact inverse $B_{\infty}$-quasi-isomorphims. As a consequence, 
$$\mathrm{MC}(\pi'): \mathrm{MC}(\tw(\AAA_0); k[[t]]) \lra \mathrm{MC}(\AAA_0; k[[t]])$$
is a bijection with inverse necessarily given by $\mathrm{MC}(s_{\delta})$. Hence, every choice $\delta$ yields a different interpretation of a representative $\phi \in \mathrm{MC}(\Sigma\CC(\AAA_0)[[t]])$ as a $cA_{\infty}$-deformation of $\tw(\AAA_0)$, but all these choices correspond to gauge equivalent solutions of $\mathrm{MC}(\Sigma \CC (\tw(\AAA_0)[[t]])$.

\subsection{The curvature problem and torsion Morita deformations}\label{parmortor}
Put $R = k[[t]]$.
Let $\AAA_0 = (\AAA_0, \mu^{(0)})$ be a small $k$-linear $A_{\infty}$-category and let
$$\mathrm{MC}(A_0; k[[t]]) = \{ \phi \in \Sigma\CC(A_0)[[t]]^1 \,\,| \,\, [\mu^{(0)}, \phi t] + \phi t \bullet \phi t = 0 \}/ \sim$$
be as before.
The natural ``curvature'' projections
$$c_A: \Sigma\CC(\AAA_0) \lra \Sigma\AAA_0(A,A): \xi \longmapsto c_A(\xi)$$
give rise to a curvature projection
$$c: \Sigma\CC(\AAA_0)[[t]] \lra \prod_{A \in \AAA_0}\Sigma\AAA_0(A,A)[[t]]: \phi \longmapsto c(\phi)= (c_A(\phi))_A$$
with $$c_A(\sum_{n = 0}^{\infty} \phi^{(n)} t^n) = \sum_{n = 0}^{\infty} c_A(\phi^{(n)}) t^n.$$  
An element $\sum_{n = 0}^{\infty} \phi^{(n)} t^n \in \mathrm{MC}(\AAA_0; k[[t]])$ corresponds to a $cA_{\infty}$-structure $\mu = \mu^{(0)} + \phi t = \mu^{(0)} + \sum_{n = 0}^{\infty}\phi^{(n)}t^{n+1}$ on ${\AAA} = R \hat{\otimes}_k \AAA_0$, and gauge equivalence of solutions corresponds to isomorphism of curved deformations.
If $\phi$ is such that $c(\phi) = 0$, then the corresponding $cA_{\infty}$-structure $\mu = \mu^{(0)} + \phi t$ is in fact an $A_{\infty}$-structure on $\AAA$. 

The curvature projection $c$ is a morphism of dg Lie algebras for the abelian dg Lie structure on the codomain. The kernel of $c$ is given by the shifted truncated Hochschild complex $\Sigma\CC_{\mathrm{tr}}(\AAA_0)[[t]]$ which controlls $A_{\infty}$-deformations of $\AAA_0$. 

Although it may be tempting to work with the truncated Hochschild complex from the point of view of deformations, this is not a good idea from the point of view of invariance properties. Indeed, in contrast with the full Hochschild complex, the truncated Hochschild complex is not invariant under Morita equivalence. In fact, we may use this observation in order to try to eliminate curvature in the following way. Suppose we are given an element $\phi \in \mathrm{MC}(\AAA_0; k[[t]])$ with $c(\phi) \neq 0$, then we may look for a Morita equivalent category $\BBB_0$ such that the corresponding element $\psi \in \mathrm{MC}(\BBB_0; k[[t]])$ satisfies $c(\psi) = 0$. The remainder of this section is devoted to making this idea precise.

The following construction is crucial. Consider $\phi \in \mathrm{MC}(\AAA_0; k[[t]])$ with corresponding $cA_{\infty}$-structure $\mu = \mu^{(0)} + \phi t$ on $\AAA$. For every $A \in \AAA_0$, we consider the element $c_A(\phi^{(0)}) \in \Sigma\AAA_0(A,A)^1$ and the twisted object $A' \in \tw(\AAA_0)$ given by $A' = A \oplus \Sigma^{-1} A$ with 
$$\delta_{A'} = \begin{pmatrix} 0 & c_A(\phi^{(0)}) \\ 0 & 0 \end{pmatrix},$$
which is $\cone(c_A(\phi^{(0)})) \in \tw(\AAA_0)$.
We also consider the element $c_A(\phi) \in \Sigma\AAA_0(A,A)[[t]]^1$ and the twisted object $A_t \in \tw(\AAA)$ given by $A_t = A \oplus \Sigma^{-1} A$ with
$$\delta_{A_t} = \begin{pmatrix} 0 & c_A(\phi) \\ - t & 0 \end{pmatrix}.$$
Let ${\AAA}_t \subseteq \tw({\AAA})$ be the full dg subcategory spanned by the objects ${A}_t$ for $A \in \AAA$ and let $\AAA' \subseteq \tw(\AAA_0)$ be the full subcategory spanned by the objects $A'$ for $A \in \AAA$. 

We obtain a morphism of dg Lie algebras
$$\mathrm{embr}_{\delta'}: \Sigma \CC(\AAA_0)[[t]] \lra \Sigma \CC(\tw(\AAA_0)[[t]] \lra \Sigma \CC(\AAA')[[t]].$$
According to \S \ref{partransp}, the image $\mathrm{embr}_{\delta'}(\phi) \in MC(\AAA'; k[[t]])$ corresponds to the $A_{\infty}$-deformation $\AAA_t$ of $\AAA'$.

\begin{lemma}\label{lemnilp}
If $c_A(\phi^{(0)}) \in \Sigma\AAA_0(A,A)^1$ is nilpotent in $H^{\ast}\Sigma\AAA_0(A,A)$, there is a canonical Morita equivalence $\AAA_0 \cong \AAA'$ with corresponding $B_{\infty}$-quasi-isomorphism $\mathrm{embr}_{\delta'}$ and bijection $\mathrm{MC}(\mathrm{embr}_{\delta'}): \mathrm{MC}(\AAA_0; k[[t]]) \cong \mathrm{MC}(\AAA'; k[[t]])$. Under this bijection, the $cA_{\infty}$-deformation $\AAA$ of $\AAA_0$ corresponds to the $A_{\infty}$-deformation $\AAA_t$ of $\AAA'$.
\end{lemma}

\begin{proof}
Since the elements $c_A(\phi^{(0)})$ are nilpotent in the derived category of $\AAA_0$, $\AAA'$ is canonically Morita equivalent to $\AAA_0$ by the octahedral axiom (see \cite[Proposition 3.16]{kellerlowen}).
\end{proof}

Now that we have found a way to replace curved deformations by uncurved deformations (under the nilpotency hypothesis on $c_A(\phi^{(0)})$), we should investigate to which extent such a replacement is well defined. In order to arrive at the statement in Theorem \ref{thmmain}, we make use of torsion Morita deformations. 
For technical convenience, in the sequel we suppose $\AAA_0$ is a dg category.

\begin{definition}
Let $\AAA_0$ be a small $k$-linear dg category.
\begin{enumerate}
\item  A \emph{(torsion) Morita $R$-deformation} of $\AAA_0$ is given by a $k$-linear dg category $\BBB_0$, a Morita $\BBB_0$-$\AAA_0$-bimodule $M_0$, and an $R$-cofibrant (torsion - in the sense of Definition \ref{deftorsion}) dg category $\BBB$ with $\BBB_0 = k \otimes_R \BBB$.
\item Let $\BBB$ with $\BBB_0$-$\AAA_0$-bimodule $M_0$ and $\CCC$ with $\CCC_0$-$\AAA_0$-bimodule $N_0$ be two Morita $R$ deformations of $\AAA_0$. 
A \emph{(torsion) Morita equivalence of deformations} between $\BBB$ and $\CCC$ is given by an $R$-cofibrant $\BBB$-$\CCC$ (torsion) Morita bimodule $M$ with $k \otimes_R M = N^{-1}_0 \otimes_{\AAA_0} M_0$.
\end{enumerate}
\end{definition}

Let $\Def_{\mathrm{tMor}}(\AAA_0;R)$ be the set of Morita deformations of $\AAA_0$ up to torsion Morita equivalence of deformations and let $\Def_{\mathrm{ttMor}}(\AAA_0;R)$ be the set of torsion Morita deformations of $\AAA_0$ up to Morita equivalence of deformations. Note that for torsion Morita deformations, a torsion Morita equivalence of deformations is the same thing as a Morita  equivalence of deformations, so we have a natural injection
$$\Theta_0: \Def_{\mathrm{ttMor}}(\AAA_0;R) \lra \Def_{\mathrm{tMor}}(\AAA_0;R).$$

\begin{lemma}
The map $\Theta_0$ is a bijection.
\end{lemma}

\begin{proof}
This follows from Proposition \ref{propattor}.
\end{proof}

Next we define a map
$$\Theta: \Def_{\mathrm{tMor}}(\AAA_0; k[[t]]) \lra \mathrm{MC}(\AAA_0; k[[t]])$$
in the following way. Consider a Morita $R$-deformation $\BBB$ with $\BBB_0$-$\AAA_0$-bimodule $M_0$. The dg structure on $\BBB$ naturally defines an element $\phi_{\BBB} \in \mathrm{MC}(\BBB_0; k[[t]])$. The Morita bimodule $M_0$ induces a canonical isomorphism
$$\Phi_{M_0}: \mathrm{MC}(\BBB_0; k[[t]]) \lra \mathrm{MC}(\AAA_0; k[[t]]).$$
We define
$$\Theta(\BBB, M_0) = \Phi_{M_0}(\phi_{\BBB}).$$

\begin{theorem} \label{thmmain}
\begin{enumerate}
\item We obtain a well defined injection
$$\Theta: \Def_{\mathrm{tMor}}(\AAA_0; k[[t]]) \lra \mathrm{MC}(\AAA_0; k[[t]]).$$
\item If $\phi = \Sigma_{n \geq 0}\phi^{(n)}t^{n} \in \mathrm{MC}(\AAA_0; k[[t]])$ is such that $c_A(\phi^{(0)}) \in \Sigma\AAA_0(A,A)^1$ is nilpotent in $H^{\ast}\Sigma\AAA_0(A,A)$ for every $A \in \AAA$, then $\phi$ is in the image of $\Theta$.
\item If $H^{2}\Sigma\AAA_0(A,A)$ is nilpotent in $H^{\ast}\Sigma\AAA_0(A,A)$ for all $A \in \AAA$, then $\Theta$ is bijective. In particular, this is the case if $\AAA$ has bounded above cohomology.
\end{enumerate}
\end{theorem}

\begin{proof}
(1) The question reduces to the following setup. Consider $R$-cofibrant dg categories $\AAA$ and $\BBB$ with a Morita $\AAA_0$-$\BBB_0$-bimodule $M_0$. Let $\phi_{\AAA} \in \mathrm{MC}(\AAA_0;[[t]])$ and $\phi_{\BBB} \in \mathrm{MC}(\BBB_0;[[t]])$ be the corresponding Maurer Cartan elements. Let $\CCC_0 = (\BBB_0 \rightarrow_{M_0} \AAA)$ be the arrow category. We have quasi-isomorphisms of $B_{\infty}$-algebras $\pi_{\AAA}: \Sigma \CC(\CCC_0) \lra \Sigma \CC(\AAA_0)$ and $\pi_{\BBB}: \Sigma \CC(\CCC_0) \lra \CC(\BBB_0)$ such that $\alpha: \mathrm{MC}(\AAA_0; k[[t]]) \cong \mathrm{MC}(\BBB_0; k[[t]])$ is induced by $\pi_{\BBB} \pi_{\AAA}^{-1}$. 

Suppose first that $\AAA$ and $\BBB$ are torsion Morita equivalent through an $R$-cofibrant $\AAA$-$\BBB$-bimodule $M$ with $k \otimes_R M = M_0$. Then the arrow category $\CCC = (\BBB \rightarrow_{M} \AAA)$ is an $R$-cofibrant dg category with $k \otimes_R \CCC = \CCC_0$ and determines an element $\phi_{\CCC} \in \mathrm{MC}(\CCC_0; k[[t]])$ with $\pi_{\AAA}(\phi_{\CCC}) = \phi_{\AAA}$ and $\pi_{\BBB}(\phi_{\CCC}) = \phi_{\BBB}$. It follows that $\phi_{\AAA}$ and $\phi_{\BBB}$ correspond through $\alpha$. This proves that $\Theta$ is well defined.

Suppose next that $\phi_{\AAA}$ and $\phi_{\BBB}$ correspond through $\alpha$. Then there exists $\phi_{\CCC} \in \mathrm{MC}(\CCC_0; k[[t]])$ with $\pi_{\AAA}(\phi_{\CCC}) = \phi_{\AAA}$ and $\pi_{\BBB}(\phi_{\CCC}) = \phi_{\BBB}$. Let $\CCC$ be the corresponding $cA_{\infty}$-deformation of $\CCC_0$. Then $\Ob(\CCC) = \Ob(\AAA) \coprod \Ob(\BBB)$ and the $cA_{\infty}$ structure reduces to the given dg structures on $\AAA$ and $\BBB$. In particular, it is an $A_{\infty}$-structure and we obtain an $\AAA$-$\BBB$ $A_{\infty}$-bimodule $M$ with $M(B,A) = \CCC(B,A)$ for $A \in \AAA$ and $B \in \BBB$. Without loss of generality, we may suppose that $M$ is an $R$-cofibrant $\AAA$-$\BBB$ dg bimodule with $k \otimes_R M = M_0$. By Proposition \ref{proptorlift}, $M$ is a torsion Morita bimodule. This proves that $\Theta$ is injective.

(2) Let $\phi$ be as in described under (2). Let $\psi = \mathrm{MC}(\mathrm{embr}_{\delta'})(\phi)$ be as in Lemma \ref{lemnilp}, with associated $A_{\infty}$-deformation $\AAA_t$ of $\AAA'$. Replacing  $\AAA_t$ with a quasi-equivalent $R$-cofibrant dg category $\tilde{\AAA}_t$, we arrive at a Morita deformation which gets mapped to $\phi$ under $\Theta$ by Lemma \ref{lemnilp}.

(3) is immediate from (1) and (2).
\end{proof}

\begin{example}
Let $X$ be a quasi-compact quasi-separated scheme. It was shown in \cite{bondalvandenbergh} that the derived category $D_{\mathrm{Qch}}(X)$ is compactly generated by a single perfect complex $M$, and thus equivalent to the derived category $D(A_0)$ of $A_0 = \RHom(M,M)$. The dg algebra $A_0$ is cohomologically bounded above, whence Theorem \ref{thmmain}(2) applies.
\end{example}

We end this section with a brief discussion of the situation for infinitesimal deformations. Roughly speaking, we have that $n+1$-th order curved deformations give rise to $n$-th order uncurved Morita deformations. Let $(\AAA_0, \mu^{(0)})$ be a small $k$-linear dg category as before For $\phi = \sum_{k = 0}^{n-1} \phi^{(k)}t^k \in \Sigma \CC(\AAA_0)[t]/t^n$, we consider $\phi t = \sum_{k = 0}^{n-1} \phi^{(k)}t^{k+1} \in \Sigma \CC(\AAA_0)[t]/t^{n+1}$. Put 
$$\mathrm{MC}(\AAA_0; k[t]/t^{n+1}) = \{ \phi \in (\Sigma \CC(\AAA_0)[t]/t^n)^1 \,\,| \,\, [\mu^{(0)}, \phi t] + \phi t \bullet \phi t = 0 \}/ \sim,$$
the set of elements $\phi$ for which $\phi t \in \Sigma \CC(\AAA_0)[t]/t^{n+1}$ is a solution of the Maurer Cartan equation up to gauge equivalence of the elements $\phi t$. Let $\Def_{\mathrm{Mor}}(\AAA_0; k[t]/t^{n+1})$ be the set of $n$-th order Morita deformations of $\AAA_0$ up to Morita equivalence.
We obtain a map
\begin{equation}
\Theta_n: \Def_{\mathrm{Mor}}(\AAA_0; k[t]/t^{n+1}) \lra \mathrm{MC}(\AAA_0; k[t]/t^{n+1}).
\end{equation}

\begin{theorem}\label{thmmaininf}
Suppose $\AAA_0$ is cohomologically bounded above.
\begin{enumerate}
\item The map $\Theta_n$ is injective.
\item If $\phi  \in \mathrm{MC}(\Sigma \CC(\AAA_0)[t]/t^{n+1})$ is in the image of $\mathrm{MC}(\Sigma \CC(\AAA_0)[t]/t^{n+2})$, then $\phi$ is in the image of $\Theta_n$.
\end{enumerate}
\end{theorem}

\begin{proof}
This is a variant of the proof of Theorem \ref{thmmain}, this time based upon Proposition \ref{propinfdef}.
\end{proof}

\subsection{The graded field}\label{pargraded}
Let $k$ be a field and let $A_0 = k[x, x^{-1}]$ be the graded field with $\deg(x) = 2$ (see \cite{kellerlowen}). Put $A= A_0[[t]]$. For every $n \in \N$, the element $$\phi_{(n)} = xt^n \in \Sigma A_0[[t]]^1 \subseteq \Sigma\CC^1(A_0)[[t]]$$
is such that $\phi_{(n)} t = xt^{n+1}$ is a solution of the Maurer-Cartan equation since $A$ is super-commutative. We obviously have 
$$c(\phi_{(n)}^{(0)}) = \phi_{(n)}^{(0)} = \begin{cases} x \,\, &\text{if} \,\, n = 0; \\ 0 \,\, & \text{otherwise.} \end{cases}$$
For $n \neq 0$, $c(\phi_{(n)}^{(0)})$ is thus nilpotent in $H^*\Sigma A_0 = \Sigma A_0$ and Theorem \ref{thmmain}(2) applies. The Morita deformation of $A_0$ corresponding to $\phi_{(n)}$ is obtained as follows. The object $A' \in \tw(A_0)$ is given by $A' = A \oplus \Sigma^{-1}A$ with $\delta_{A'} = 0$, and the corresponding endomorphism dg algebra $\mathrm{End}(A')$ is a matrix algebra (with zero differential). The object $A_t \in \tw(A)$ is given by $A_t = A \oplus \Sigma^{-1}A$ with 
$$\delta_{A_t} = \begin{pmatrix} 0 & xt^n \\ - t & 0 \end{pmatrix}.$$
The corresponding endomorphism dg algebra $\mathrm{End}(A_t)$ is a matrix algebra whose differential is given by the super-commutator with $\delta_{A_t}$.

For $n = 0$, we have a similar description of $A_t$, but this time $A' = A \oplus \Sigma^{-1}A$ is endowed with
$$\delta_{A'} = \begin{pmatrix} 0 & x \\  0 & 0 \end{pmatrix},$$
that is $A' \cong \cone(x)$. Since $x \in A_0$ is invertible (with inverse $x^{-1}$), it is certainly not nilpotent so Theorem \ref{thmmain}(2) does not apply. In fact, $\cone(x)$ is a null-homotopic object in $\tw(A_0)$, a null-homotopy being given by 
$$\delta_{A'} = \begin{pmatrix} 0 & 0 \\  x^{-1} & 0 \end{pmatrix}.$$
Consequently, the endomorphism dg algebra $\mathrm{End}(A')$ is null-homotopic too, and thus quasi-equivalent to the zero category. In particular, $A$ is not Morita equivalent to $\mathrm{End}(A')$, and $\mathrm{End}(A_t)$ does not constitute a deformation of $A$ in any reasonable sense.

Let us show that the element $\phi_{(0)} = x$ is not in the image of the map $\Theta$. Suppose on the contrary that it is in the image. We start from the $cA_{\infty}$-deformation $A = A_0[[t]]$ endowed with $$m = m_2^{(0)} + x\epsilon$$ where $m_2^{(0)}$ is the multiplication of $A_0$. Up to quasi-equivalence, every category $\BBB_0$ Morita-equivalent to $A_0$ can be realized as a full subcategory of the dg model $\Tw_{ut}(A_0)$ of the derived category $D(A_0)$, and for such $\BBB_0$, a $cA_{\infty}$-deformation $\BBB = \BBB_0[[t]]$ of $\BBB_0$ corresponding to a gauge equivalent solution of the Maurer Cartan equation is constructed as follows (see \S \ref{partransp}). For every object $(M, \delta^{(0)}) \in \BBB_0 \subseteq \Tw_{ut}(A_0)$, there is a connection
$$\delta = \sum_{n = 0}^{\infty} \delta^{(n)} t^n$$
for which $(M, \delta)$ is uncurved. We thus have
$$\begin{aligned}
0 = c_M & = m^{(0)}_2(\delta^{(0)}, \delta^{(0)}) \\ 
& + [m^{(0)}_2(\delta^{(0)}, \delta^{(1)}) + m^{(0)}_2(\delta^{(1)}, \delta^{(0)}) + x]t \\
& + \sum_{n = 2}^{\infty} [\sum_{p + q = n} m^{(0)}_2(\delta^{(p)}, \delta^{(q)})]t^n.
\end{aligned}$$
We claim that $(M, \delta^{(0)})$ is null-homotopic. Put
$$h = x^{-1}\delta^{(1)} \in \Hom^{-1}_0(M,M).$$
We have
$$\begin{aligned}
\embr_{\delta^{(0)}}(m^{(0)}_2)(h) & = m^{(0)}_2(\delta^{(0)}, h) + m^{(0)}_2(h, \delta^{(0)}) \\
& = x^{-1}[m^{(0)}_2(\delta^{(0)}, \delta^{1)}) + m_2^{(0)}(\delta^{(1)}, \delta^{(0)})]\\
& = x^{-1} x = 1.
\end{aligned}$$
As before, it then follows that $\BBB_0$ is quasi-equivalent to the zero category, in contradiction with the fact that $\BBB_0$ is Morita equivalent to $A_0$.

Note that from perturbation theory, it follows that $(M, \delta)$ is null-homotopic as well, and the only uncurved deformation one can obtain from (subcategories of) $\Tw_{ut}(A_0)$ is a zero category. For a precise discussion in the case of first order deformations, see \cite{dedekenlowen2}. For other arguments why the only uncurved ``deformation'' of the graded field is the zero category, we refer to \cite{kellerlowennicolas}.

\def\cprime{$'$} \def\cprime{$'$}
\providecommand{\bysame}{\leavevmode\hbox to3em{\hrulefill}\thinspace}
\providecommand{\MR}{\relax\ifhmode\unskip\space\fi MR }
\providecommand{\MRhref}[2]{%
  \href{http://www.ams.org/mathscinet-getitem?mr=#1}{#2}
}
\providecommand{\href}[2]{#2}


\begin{thebibliography}{10}

\bibitem{alonsojeremiaslipman}
L.~Alonso~Tarr{\'{\i}}o, A.~Jerem{\'{\i}}as~L{\'o}pez, and J.~Lipman,
  \emph{Local homology and cohomology on schemes}, Ann. Sci. \'Ecole Norm. Sup.
  (4) \textbf{30} (1997), no.~1, 1--39. \MR{1422312 (98d:14028)}

\bibitem{bondalvandenbergh}
A.~Bondal and M.~Van~den Bergh, \emph{Generators and representability of
  functors in commutative and noncommutative geometry}, Mosc. Math. J.
  \textbf{3} (2003), no.~1, 1--36, 258. \MR{MR1996800 (2004h:18009)}

\bibitem{bondalkapranov}
A.~I. Bondal and M.~M. Kapranov, \emph{Framed triangulated categories}, Mat.
  Sb. \textbf{181} (1990), no.~5, 669--683. \MR{MR1055981 (91g:18010)}

\bibitem{bourbaki2}
N.~Bourbaki, \emph{General topology. {C}hapters 5--10}, Elements of Mathematics
  (Berlin), Springer-Verlag, Berlin, 1998, Translated from the French, Reprint
  of the 1989 English translation. \MR{1726872 (2000h:54001b)}

\bibitem{dedekenlowen2}
O.~De~Deken and W.~Lowen, \emph{On deformations of triangulated models}, Adv.
  Math. \textbf{243} (2013), 330--374. \MR{3062749}

\bibitem{fukaya}
K.~Fukaya, \emph{Deformation theory, homological algebra and mirror symmetry},
  Geometry and physics of branes (Como, 2001), Ser. High Energy Phys. Cosmol.
  Gravit., IOP, Bristol, 2003, pp.~121--209. \MR{MR1950958 (2004c:14015)}

\bibitem{getzlerjones}
E.~Getzler and J.~D.~S. Jones, \emph{Operads, homotopy algebra and iterated
  integrals for double loop spaces}, preprint hep-th/9403055.

\bibitem{kashiwarashapira}
M.~Kashiwara and P.~Schapira, \emph{Deformation quantization modules},
  Ast\'erisque (2012), no.~345, xii+147. \MR{3012169}

\bibitem{keller}
B.~Keller, \emph{On the cyclic homology of exact categories}, J. Pure Appl.
  Algebra \textbf{136} (1999), no.~1, 1--56. \MR{MR1667558 (99m:18012)}

\bibitem{kellerlowen}
B.~Keller and W.~Lowen, \emph{On {H}ochschild cohomology and {M}orita
  deformations}, Int. Math. Res. Not. IMRN (2009), no.~17, 3221--3235.
  \MR{2534996 (2010j:16023)}

\bibitem{kellerlowennicolas}
B.~Keller, W.~Lowen, and P.~Nicol{\'a}s, \emph{On the (non)vanishing of some
  ``derived'' categories of curved dg algebras}, J. Pure Appl. Algebra
  \textbf{214} (2010), no.~7, 1271--1284. \MR{2587002 (2011b:16041)}

\bibitem{lefevre}
K.~Lef{\`e}vre-Hasegawa, \emph{Sur les {$A\sb \infty$}-cat{\'e}gories},
  Th{\`e}se de doctorat, Universit{\'e} Denis Diderot Paris 7, November 2003,
  available at the homepage of B. Keller.

\bibitem{lowen9}
W.~Lowen, \emph{Hochschild cohomology, the characteristic morphism and derived
  deformations}, Compos. Math. \textbf{144} (2008), no.~6, 1557--1580.
  \MR{2474321 (2009m:18016)}

\bibitem{petit}
F.~Petit, \emph{D{G} affinity of {DQ}-modules}, Int. Math. Res. Not. IMRN
  (2012), no.~6, 1414--1438. \MR{2899956}

\bibitem{PSY2}
M.~Porta, L.~Shaul, and A.~Yekutieli, \emph{Completion by derived double
  centralizer}, Algebr. Represent. Theory \textbf{17} (2014), no.~2, 481--494.
  \MR{3181733}

\bibitem{portashaulyekutieli}
\bysame, \emph{On the homology of completion and torsion}, Algebr. Represent.
  Theory \textbf{17} (2014), no.~1, 31--67. \MR{3160712}

\bibitem{positselski3}
L.~Positselski, \emph{Weakly curved {$A\sb \infty$}-algebras over a topological
  local ring},  (2012), arXiv:1202.2697.

\bibitem{positselski:contra}
\bysame, \emph{Contramodules},  (2015), arXiv:1503.00991.

\bibitem{positselski:MGM}
\bysame, \emph{Dedualizing complexes and {MGM} duality},  (2015),
  arXiv:1503.05523.

\bibitem{yekutieli}
A.~Yekutieli, \emph{The continuous {H}ochschild cochain complex of a scheme},
  Canad. J. Math. \textbf{54} (2002), no.~6, 1319--1337. \MR{MR1940241
  (2004d:16016b)}

\bibitem{yekutieli4}
\bysame, \emph{On flatness and completion for infinitely generated modules over
  {N}oetherian rings}, Comm. Algebra \textbf{39} (2011), no.~11, 4221--4245.
  \MR{2855123 (2012k:13058)}

\bibitem{yekutieli:MC}
\bysame, \emph{M{C} elements in pronilpotent {DG} {L}ie algebras}, J. Pure
  Appl. Algebra \textbf{216} (2012), no.~11, 2338--2360. \MR{2927171}

\end{thebibliography}
\end{document}